\documentclass[a4paper, 12pt]{amsart}
\calclayout

\usepackage{amsmath,amsthm,amssymb,mathrsfs,mathtools} 
\usepackage[utf8]{inputenc}
\usepackage[
backend=biber,
style=alphabetic,
sorting=nyt,
backref=true,
url=false,
isbn=false,
doi=false]{biblatex}
\usepackage{bm}
\usepackage{enumitem}
\usepackage{tensor}
\usepackage{tikz-cd}
\usepackage{hyperref}
\hypersetup{
  colorlinks=true,
  linkcolor={red!50!black},
  filecolor=magenta,
  urlcolor=cyan,
  citecolor={blue!80!black}
}
\theoremstyle{plain}

\newtheorem{theorem}     [equation]  {Theorem}
\newtheorem{proposition} [equation]  {Proposition}
\newtheorem{lemma}       [equation]  {Lemma}
\newtheorem{corollary}   [equation]  {Corollary}

\newtheorem*{theorem*}               {Theorem}
\newtheorem*{proposition*}           {Proposition}
\newtheorem*{lemma*}                 {Lemma}
\newtheorem*{corollary*}             {Corollary}

\theoremstyle{definition}

\newtheorem{definition/} [equation]  {Definition}
\newenvironment{definition}{ \pushQED{\qed}\begin{definition/}} {\popQED\end{definition/}}

\newtheorem*{definition*/}           {Definition}
\newenvironment{definition*}{ \pushQED{\qed}\begin{definition*/}} {\popQED\end{definition*/}}

\theoremstyle{remark}

\newtheorem{remark/}     [equation]  {Remark}
\newenvironment{remark}{ \pushQED{\qed}\begin{remark/}} {\popQED\end{remark/}}
\newtheorem{example/}    [equation]  {Example}
\newenvironment{example}{ \pushQED{\qed}\begin{example/}} {\popQED\end{example/}}

\newtheorem*{remark*/}               {Remark}
\newenvironment{remark*}{ \pushQED{\qed}\begin{remark*/}} {\popQED\end{remark*/}}
\newtheorem*{example*/}              {Example}
\newenvironment{example*}{ \pushQED{\qed}\begin{example*/}} {\popQED\end{example*/}}

\newtheorem*{question*}              {Question}

\numberwithin{equation}{section}

\newenvironment{enumeratea}{
  \begin{enumerate}[label = (\alph*)]}{
  \end{enumerate}}

\makeatletter 
\def\l@subsection{\@tocline{2}{0pt}{3pc}{6pc}{}} \makeatother

\setlist[description]{labelindent=\parindent,font=\normalfont}

\makeatletter
\@namedef{subjclassname@2020}{\textup{2020} Mathematics Subject Classification}
\makeatother

\newcommand{\ol}[1]{\overline{#1}}

\newcommand{\cl}[1]{\mathcal{#1}}
\newcommand{\sr}[1]{\mathscr{#1}}

\newcommand{\mbf}[1]{\mathbf{#1}}
\newcommand{\msf}[1]{\mathsf{#1}}

\newcommand{\ti}[1]{\tilde{#1}}

\newcommand{\rra}{\rightrightarrows}

\newcommand{\fiber}[2]{\tensor[_{#1}]{\times}{_{#2}}}
\newcommand{\germ}[2]{[#1]_{{#2}}}

\newcommand{\define}[1]{\textbf{#1}}

\newcommand{\R}{\mathbb{R}}

\newcommand{\Z}{\mathbb{Z}}

\newcommand{\ext}{\text{ext}}

\DeclareMathOperator{\Aff}{Aff}
\DeclareMathOperator{\dom}{dom}
\DeclareMathOperator{\Diff}{Diff}
\DeclareMathOperator{\Hol}{Hol}
\DeclareMathOperator{\hol}{hol}
\DeclareMathOperator{\Eff}{Eff}

\DeclareMathOperator{\id}{id}
\DeclareMathOperator{\Iso}{Iso}

\DeclareMathOperator{\pr}{pr}

\begin{document}


\title{Lie groupoids determined by their orbit spaces}
\author{David Miyamoto}
\address{David Miyamoto \\
  Max Planck Institute for Mathematics \\
  Vivatsgasse 7, 53111 Bonn, Germany}
\email{miyamoto@mpim-bonn.mpg.de}

\subjclass[2020]{58H05, 57R30, 57P05}
 \keywords{Lie groupoids, diffeology, quasifolds, submersions}

\keywords{}
\date{}


 \begin{abstract}
   Given a Lie groupoid, we can form its orbit space, which carries a natural diffeology. More generally, we have a quotient functor from the Hilsum-Skandalis category of Lie groupoids to the category of diffeological spaces. We introduce the notion of a lift-complete Lie groupoid, and show that the quotient functor  restricts to an equivalence of the categories: of lift-complete Lie groupoids with isomorphism classes of surjective submersive bibundles as arrows, and of quasi-\'{e}tale diffeological spaces with surjective local subductions as arrows. In particular, the Morita equivalence class of a lift-complete Lie groupoid, alternatively a lift-complete differentiable stack, is determined by its diffeological orbit space. Examples of lift-complete Lie groupoids include quasifold groupoids and \'{e}tale holonomy groupoids of Riemannian foliations.
 \end{abstract}
 
\maketitle
 
\tableofcontents

\section{Introduction}
\label{sec:introduction}

Lie groupoids capture symmetries of various geometric situations. For instance, if $\{U_i\}$ is a cover of a manifold $M$, we form the covering groupoid $\bigsqcup_{i,j} U_i \cap U_j \rra \bigsqcup_i U_i$, which records how to assemble $M$ by gluing the $U_i$; if $G$ is a Lie group acting smoothly on $M$, we form the action groupoid $G \ltimes M \rra M$, which encodes the action; and if $\cl{F}$ is a regular foliation on $M$, we form the holonomy groupoid $\Hol(\cl{F}) \rra M$, which describes the transverse geometry. By collecting the points of $M$ attainable by the symmetries imposed by a Lie groupoid $G \rra M$, we get the set of orbits $M/G$.

This rarely inherits a manifold structure, and even its topology may be trivial. This happens for the Kronecker foliation of the 2-torus, given by lines of irrational slope. However, $M/G$ inherits a diffeology from $M$, which is more robust: the leaf space of the Kronecker foliation has non-trivial diffeology.

A diffeology is a generalized smooth structure on a set $X$, introduced by Souriau in the 1980s \cite{Sou79}, though a similar structure was earlier used by Chen \cite{Che77}. Diffeological spaces, and smooth maps between them, form a category $\mbf{Diffeol}$. The category of manifolds embeds fully and faithfully into into $\mbf{Diffeol}$, and $\mbf{Diffeol}$ is closed under quotients.\footnote{It is a quasi-topos \cite{BH11}.} Therefore we can equip the orbit space $M/G$ with its quotient diffeology.

We could alternatively view the Lie groupoid itself as encoding the smooth structure of $M/G$. In analogy with the fact a manifold structure is an equivalence class of smooth atlases, we say Lie groupoids determine the same transverse structure when they are Morita equivalent. We express this in the bicategory $\mbf{Bi}$ of Lie groupoids, principal bibundles, and isomorphisms of bibundles, where Lie groupoids are Morita equivalent if and only if they are isomorphic in $\mbf{Bi}$. We form the Hilsum-Skandalis (1-)category $\mbf{HS}$ by identifying isomorphic principal bibundles and forgetting the 2-arrows.

A Morita equivalence class of Lie groupoids may be viewed as a differentiable stack, denoted $[M/G]$. More precisely, we have the 2-category $\mbf{St}$ of differentiable stacks, and an equivalence of bicategories $\mbf{Bi} \to \mbf{St}$. Morita equivalent Lie groupoids correspond to isomorphic stacks, and thus $[M/G]$ is another smooth model for $M/G$.

There is a quotient functor $\mbf{F}:\mbf{Bi} \to \mbf{Diffeol}$, taking $G$ to $M/G$. This is not an equivalence of categories. In \cite{KM22}, we showed with Karshon that $\mbf{F}$ does induce an equivalence when restricted to quasifolds, which are (as spaces or Lie groupoids) modelled by affine actions of countable groups on Cartesian spaces,\footnote{$\R^n$ for some $n \geq 0$.} and to local isomorphisms.
\begin{theorem*}[\cite{KM22}]
  The quotient functor $\mbf{F}$ restricts to a functor from the bicategory of effective quasifold groupoids, locally invertible bibundles, and isomorphisms of bibundles, to the category of diffeological quasifolds and local diffeomorphisms. This restriction is:
  \begin{itemize}
  \item essentially surjective on objects;
  \item full on 1-arrows;
  \item faithful up to isomorphism of 1-arrows.
  \end{itemize}
  Furthermore, the stack represented by an effective quasifold groupoid is determined by its diffeological orbit space, and $\mbf{F}$ gives an equivalence of categories if we descend to $\mbf{HS}$.
\end{theorem*}
By ``determined,'' we mean that effective quasifold groupoids represent isomorphic stacks if and only if their orbit spaces are diffeomorphic. This theorem applies \emph{mutatis mutandis} to orbifold groupoids and diffeological orbifolds. We extend this result to two pairs of categories.
\begin{description}
\item[(A)] The sub-bicategory $\mbf{LiftComp}^{\twoheadrightarrow}$ of $\mbf{Bi}$, consisting of lift-complete (assumed to be \'{e}tale and effective) Lie groupoids and surjective submersive bibundles between them, and the subcategory $\mbf{QUED}^{\twoheadrightarrow}$ of $\mbf{Diffeol}$, consisting of quasi-\'{e}tale diffeological spaces and surjective local subductions between them.
\item[(B)] The sub-bicategory $\mbf{QLiftComp}$ of $\mbf{LiftComp}$, consisting of lift-complete Lie groupoids whose orbit map is a Q-chart, and the sub-category $\mbf{QMan}$ of $\mbf{QUED}$ consisting of Q-manifolds.
\end{description}

\begin{theorem}
  \label{thm:1}
  The quotient functor $\mbf{F}$ restricts in two ways:
  \begin{description}
  \item[(A)] To $\mbf{F}_{(A)}:\mbf{LiftComp}^{\twoheadrightarrow} \to \mbf{QUED}^{\twoheadrightarrow}$. This restriction is:
  \begin{itemize}
  \item essentially surjective on objects;
  \item full on 1-arrows;
  \item faithful up to isomorphism of 1-arrows.
  \end{itemize}
  Furthermore, the stack represented by a lift-complete Lie groupoid is determined by its diffeological orbit space, and $\mbf{F}_{(A)}$ gives an equivalence of categories if we descend to $\mbf{HS}$.
  \item[(B)] To $\mbf{F}_{(B)}:\mbf{QLiftComp} \to \mbf{QMan}$, satisfying the same conclusions as $\mbf{F}_{(A)}$.
  \end{description}
\end{theorem}
We leave the definitions for Sections \ref{sec:diff-lie-group} and \ref{sec:lift-complete-lie-1}, and here give some examples. Effective quasifold groupoids (Proposition \ref{prop:3}) and \'{e}tale holonomy groupoids of Riemannian foliations (Section \ref{sec:exampl-riem-foli}) are lift-complete, and diffeological quasifolds and leaf spaces of Riemannian foliations are quasi-\'{e}tale diffeological spaces. Both examples are lift-complete, respectively quasi-\'{e}tale, because the associated pseudogroups consist of solutions to partial differential equations (PDEs) of a certain type -- namely complete and quasi-analytic -- which we outline in Subsection \ref{sec:lift-compl-pseud}. The category $\mbf{QLiftComp}$ contains the \'{e}tale holonomy groupoids of Kronecker foliations of the 2-torus, and their leaf spaces are Q-manifolds. The connection to Riemannian foliations gives the following corollary.
\begin{corollary}
  \label{cor:2}
  Two Riemannian foliations have diffeomorphic leaf spaces if and only if their holonomy groupoids are Morita equivalent.
\end{corollary}
This implies that Riemannian foliations have diffeomorphic leaf spaces if and only if they are transverse equivalent. We elaborate in Section \ref{sec:exampl-riem-foli}.

It is not possible to extend Theorem \hyperref[thm:1]{\ref{thm:1} (A)} to include all arrows. Even if we restrict the objects to orbifolds, Examples \ref{ex:5} and \ref{ex:6} show that $\mbf{F}_{(A)}$ can fail to be faithful up to isomorphism, or full.

Watts introduced the quotient functor $\mbf{F}:\mbf{Bi} \to \mbf{Diffeol}$ in \cite{Wat22},\footnote{Relevant work was completed in 2013.} and he showed in \cite{Wat17} that $\mbf{F}$ is full on isomorphisms when we restrict to orbifold groupoids and diffeological orbifolds. It was previously known that $\mbf{F}$ is essentially surjective and full on isomorphisms for orbifolds as defined by Satake \cite{Sat56,Sat57}, for example see \cite[Chapter 5, Section 6]{MM03}. Iglesias-Zemmour, Karshon, and Zadka give a complete comparison of the various definitions of orbifold in \cite{IZKZ10}. Quasifolds were introduced outside diffeology by Prato \cite{Pra99,Pra01}, in order to extend the Delzant theorem from symplectic geometry to non-rational simple polytopes. They were introduced to diffeology in \cite{IZP20} and \cite{KM22}. Hilsum and Skandalis first defined bibundles in \cite{HS87} as generalized morphisms between foliations, and they also describe submersive bibundles. For the bicategory of Lie groupoids, and for differentiable stacks, our main source is \cite{Ler10}, but see also \cite{BX11,Blo08}. Surjective submersions between stacks also appear in \cite{HF19,BNZ20}. For Riemannian foliations, we use the foundational material in \cite{Mol88}, and we the notions of transverse equivalence from \cite{GZ19} and \cite{Miy23}.

This article is structured as follows. In Section \ref{sec:diff-lie-group}, we review the necessary tools from diffeology (Subsection \ref{sec:diffeology}) and Lie groupoids (Subsection \ref{sec:groupoids}). We also define quasi-\'{e}tale diffeological spaces, Q-manifolds, and surjective submersive bibundles. In Section \ref{sec:lift-complete-lie-1}, we define lift-complete Lie groupoids, and collect the lemmas needed to prove our main theorem in Section \ref{sec:an-equiv-categ}. We also characterize a class of lift-complete Lie groupoids in terms of solutions of certain PDEs. Section \ref{sec:exampl-riem-foli} contains an application to Riemannian foliations.

\subsubsection*{Acknowledgements}

I would like to thank Yael Karshon for motivating the investigations that led to this result. Thanks also to Joel Villatoro and Alireza Ahmadi for helpful discussions on quasi-\'{e}tale diffeological spaces and diffeological \'{e}tale manifolds, respectively, and to Francesco Cattafi and Luca Accornero for sharing their knowledge of Lie pseudogroups. Finally, I am grateful to the Max Planck Institute for Mathematics for providing a welcoming and productive environment as I finished this paper.

\section{Diffeology and Lie groupoids}
\label{sec:diff-lie-group}

\subsection{Diffeology}
\label{sec:diffeology}

This will be a brief introduction to diffeology, and we refer to the textbook \cite{IZ13} for details. A \define{diffeological space} is a set $X$ equipped with a \define{diffeology}, which is a set of maps $\sr{D}$ from open subsets of Cartesian spaces into $X$, called \define{plots}, satisfying
\begin{itemize}
\item $\sr{D}$ contains all locally constant maps,
\item if $F:\msf{V} \to \msf{U}$ is smooth, and $p:\msf{U} \to X$ is a plot, then $p F$\footnote{We will often write composition as juxtaposition, e.g.\ $pF$ for $p \circ F$.} is a plot,
\item if $p:\msf{U} \to X$ is a map, and it is locally a plot, then $p$ is a plot.
\end{itemize}
If the need arises, we might denote a plot by $\msf{p}$ instead. Every set admits a diffeology. For instance, both
\begin{equation*}
  \{\text{all maps } \msf{U} \to X\} \text{ and } \{\text{locally constant maps } \msf{U} \to X\}
\end{equation*}
are diffeologies, called respectively \define{coarse} and \define{discrete}. If $M$ is a smooth manifold,\footnote{A Hausdorff and second-countable topological space equipped with a maximal smooth atlas.} then the set of all smooth maps $\msf{U} \to M$ is a canonical diffeology $\sr{D}_M$ on $M$.

A map $f:X \to Y$ between diffeological spaces is \define{smooth} if $f p :\msf{U} \to Y$ is a plot of $Y$ for every plot $p$ of $X$. \textbf{Diffeol} denotes the category of diffeological spaces and smooth maps between them. A map $f:M \to N$ between smooth manifolds is diffeologically smooth if and only if it is smooth in the usual sense, so the functor
\begin{equation*}
  \mathbf{Man} \to \mathbf{Diffeol}, \quad M \mapsto (M, \sr{D}_M), \quad f \mapsto f
\end{equation*}
is a fully faithful embedding.

The \define{D-topology} of a diffeological space is the finest topology in which all the plots are continuous. Smooth maps are continuous in the D-topology, and the D-topology of a manifold is its manifold topology. We say a smooth map $f:X \to Y$ is a \define{local diffeomorphism} if we can cover $X$ with D-open subsets $U$ such that $f(U)$ is D-open in $Y$, and $f:U \to f(U)$ is a smooth map (with respect to the subset diffeologies defined below) with smooth inverse.

Diffeology propagates to subsets. If $A \subseteq X$ is a subset of a diffeological space, its \define{subset diffeology} consists of all the plots of $X$ whose image is in $A$. When $A$ has discrete subset diffeology, we call $A$ \define{totally disconnected}. 

Diffeology also propagates to quotients. If ${\sim}$ is an equivalence relation on $X$, the \define{quotient diffeology} on $X/{\sim}$ consists of all maps $p:\msf{U} \to X/{\sim}$ that locally lift along the quotient $\pi:X \to X/{\sim}$ to plots of $X$. In a diagram,
\begin{equation*}
  \begin{tikzcd}
   & X \ar[d, "\pi"]\\
   \forall r\in \msf{U} \ar[ur, dashed,  "\exists q_r"] \ar[r, "p"] & X/{\sim},
\end{tikzcd}
\end{equation*}
where the dashed line indicates that $q_r$ is defined in a neighbourhood of $r$.

The quotient map $\pi:X \to X/{\sim}$ is an example of a subduction; a map $\pi:X \to Y$ is a \define{subduction} if it is surjective and every plot of $Y$ locally lifts along $\pi$ to a plot of $X$. In the sequel, we will have quotients satisfying a stronger property. First, to establish terminology, if we have a plot $p:\msf{U} \to X$, and a distinguished $r \in \msf{U}$, we will say $p$ is a plot \define{pointed} at $p(r)$. We will write $p:(\msf{U},r) \to (X, p(r))$. We will also use this language and notation for maps.
\begin{definition}
  A smooth map $\pi:X \to Y$ is a \define{local subduction} if, for every pointed plot $p:(\msf{U},r) \to (Y,p(r))$, and every $x \in X$ with $\pi(x) = p(r)$, there is a local lift of $p$ along $\pi$ pointed at $x$.
\end{definition}
  A smooth map $\pi:M \to N$ of manifolds is a local subduction if and only if it is a submersion. If $\pi$ is merely a subduction, then it may not be a submersion. A composition of local subductions is a local subduction, so we have the category $\mbf{Diffeol}^{\twoheadrightarrow}$ of diffeological spaces and surjective local subductions between them.

\begin{definition}
  \label{def:1}
  A smooth map $\pi:X \to Y$ is \define{quasi-\'{e}tale} if
  \begin{itemize}
  \item it is a local subduction,
  \item its fibers are totally disconnected, and
  \item[(QE)] if $f:U \to X$ is a smooth map from a D-open subset $U \subseteq X$, and $\pi f = \pi|_U$, then $f$ is a local diffeomorphism.
  \end{itemize}
  A diffeological space $Y$ is \define{quasi-\'{e}tale} if for every $y \in Y$, there is some quasi-\'{e}tale map $\pi:M \to Y$, with $M$ a manifold, such that $y \in \pi(M)$. 
\end{definition}

\begin{definition}
We denote by $\mbf{QUED}$ the sub-category of $\mbf{Diffeol}$ whose objects are second-countable\footnote{In Section \ref{sec:an-equiv-categ}, we will need to take a denumerable cover of a quasi-\'{e}tale space.} quasi-\'{e}tale diffeological spaces. When we only take surjective local subductions as arrows, we use $\mbf{QUED}^{\twoheadrightarrow}$.
\end{definition}
Villatoro \cite{Vill23} introduced quasi-\'{e}tale maps and spaces, and the notation $\mbf{QUED}$. Villatoro showed that the 2-category of groupoid objects of $\mbf{QUED}$ admits a Lie functor into the category of Lie algebroids, and thereby obtains an extension of Lie's third theorem to Lie algebroids. Quasifolds and leaf spaces of Riemannian foliations are quasi-\'{e}tale diffeological spaces, see Example \ref{ex:3} and Proposition \ref{prop:7}.

\begin{remark}
 Villatoro \cite[Example 3.12]{Vill23} points out that $\mbf{QUED}^{\twoheadrightarrow}$ is not closed under fiber products, and gives the example of $\pi = \ti{\pi}:\R \to \R/(x \sim -x)$, whose fiber product $\R \fiber{\pi}{\ti{\pi}} \R$ is not quasi-\'{e}tale. As a consequence, if we define groupoid objects in $\mbf{QUED}$ using local subductions for the source and target, the space of composable arrows need not be in $\mbf{QUED}$. For this reason, Villatoro avoids calling local subductions ``submersions,'' and we will follow this convention. Note that Villatoro defines submersions in $\mbf{QUED}$ as local subductions that are Cartesian in $\mbf{QUED}$.
\end{remark}

Within $\mbf{QUED}$ is the sub-category of Q-manifolds. For a smooth map $f$, let $\germ{f}{x}$ denote the germ of $f$ at $x$.
\begin{definition}
  \label{def:4}
 A quasi-\'{e}tale map $\pi:X \to Y$ is a \define{Q-chart} if
  \begin{enumerate}
    \item[(Q)] for any plots $p,\ti{p}:\msf{U} \to X$ such that $\pi p = \pi \ti{p}$, whenever $p(r) = \ti{p}(r)$, we also have $\germ{p}{r} = \germ{\ti{p}}{r}$.
    \end{enumerate}
  A diffeological space $Y$ is a \define{Q-manifold}, or a \define{diffeological \'{e}tale manifold}, if, for every $y \in Y$, there is some Q-chart $\pi:M \to Y$, with $M$ a manifold, such that $y \in \pi(M)$. We denote the subcategory of $\mbf{QUED}$ consisting of second-countable Q-manifolds by $\mbf{QMan}$.
\end{definition}
\begin{remark}
  \label{rem:8}
  We used plots in condition \hyperref[def:4]{(Q)}, following the philosophy that diffeology is built from plots. But because manifolds are locally Cartesian, it is equivalent to replace $p,\ti{p}:\msf{U} \to X$ with two maps from any manifold into $X$.
\end{remark}

Diffeological \'{e}tale manifolds were introduced by Ahmadi \cite{Ahm23}, and Q-manifolds were introduced by Barre in \cite{Bar73}. Barre did not work with diffeology, since Souriau only introduced diffeology in the 1980s. Instead, Barre defined a Q-manifold as a set $Y$ equipped with an equivalence class of surjective Q-charts, where two such charts $\pi$ and $\pi'$ are equivalent if $\pi \sqcup \pi':M \sqcup M' \to Y$ is also a Q-chart. One can show that Barre's category of Q-manifolds is equivalent to ours, by assigning to an equivalence class of surjective Q-charts the push-forward diffeology induced by any representative chart. We will use Barre's terminology in this article, because the word ``\'{e}tale'' already appears in ``quasi-\'{e}tale space.''

\begin{example}
  The diffeological quotient space $\R/(x \sim -x)$ is quasi-\'{e}tale but not a Q-manifold. For instance, $\id, -\id:\R \to \R$ are plots that descend to the same function $\R \to \R/(x \sim -x)$, and they agree at $0$, yet they have different germs at $0$. Therefore, orbifolds are not necessarily Q-manifolds.

  On the other hand, the \emph{irrational torus}, $T_\alpha := \R/(\Z + \alpha \Z)$ is a Q-manifold. To see this, suppose that $p,\ti{p}:\msf{U} \to \R$ descend to the same map, and $p(r) = \ti{p}(r)$. Let $\msf{V}$ be a relatively compact connected open neighbourhood of $r$ in $\msf{U}$. The sets
  \begin{equation*}
    \Delta_{m+\alpha n} := \{r' \in \ol{\msf{V}} \mid \ti{p}(r') = p(r') + m + \alpha n\}
  \end{equation*}
 partition $\ol{\msf{V}}$ into countably many disjoint closed sets. By Serpinski's theorem \cite[Theorem 6.1.27]{Eng89}, such a partition can have at most one component. This component must be $\Delta_0$, since it contains at least one point, $r$. Thus $\germ{p}{r} = \germ{\ti{p}}{r}$, as required. A similar argument shows that certain classes of foliations (e.g.\ Riemannian foliations whose leaves are without holonomy) have leaf spaces that are Q-manifolds. See \cite{Mei97} for examples.
\end{example}

\subsection{Lie Groupoids}
\label{sec:groupoids}

As with diffeology, we give a quick review of Lie groupoids, and refer to \cite{MM03,Ler10} for details. A \define{Lie groupoid} $G$ is a category $G \rra M$ such that
\begin{itemize}
\item (Lie) $G$ and $M$ are manifolds, except that $G$ need not be Hausdorff nor second-countable, and the source and target are submersions with Hausdorff fibers, and
\item (Groupoid) every arrow is invertible.
\end{itemize}
Unless otherwise stated, we will use $M$ for the base of $G$, and if $H$ is a Lie groupoid, its base will be $N$. We will write arrows as $g:x \mapsto y$, and composition like that of functions: if $g:x \mapsto y$ and $g':y \mapsto z$, then $g'g:x \mapsto z$. We denote $G_x := s^{-1}(x) \cap t^{-1}(x)$, and observe that this is a (possibly not second-countable) Lie group. Every Lie groupoid partitions its base into \define{orbits}, namely the sets $t(s^{-1}(x))$ for $x \in M$. We may denote an orbit by $G \cdot x$ or $x \cdot G$. We denote the space of orbits by $M/G$, and equip it with its quotient diffeology. We call a Lie groupoid \define{\'{e}tale} if $\dim G = \dim M$, in which case $s$ and $t$ are local diffeomorphisms.
\begin{remark}
  \label{rem:5}
  It is possible to equip orbits with a not-necessarily second-countable manifold structure, such that $t:s^{-1}(x) \to G \cdot x$ a principal $G_x$-bundle, and the inclusion $G \cdot x \hookrightarrow M$ is an injective immersion. We may lose second-countability because the arrow space may not be second-countable. Therefore, it is possible that this manifold structure on $G \cdot x$ does not coincide with the subset diffeology. This is often the case for germ groupoids, see Example \ref{ex:4}.
\end{remark}

A \define{map} of Lie groupoids is a smooth functor $\varphi:G \to H$. Our first example of Lie groupoids are action groupoids. Suppose that $H$ is a Lie group acting smoothly on $N$ from the right. Then the \define{action groupoid} $N \rtimes H$ has arrows $N \times H$ and base $N$. A pair $(y,h)$ is an arrow $(y,h):y\cdot h \mapsto y$, and the multiplication is $(y,h)\cdot(y',h') := (y,hh')$. A map between action groupoids is equivalent to an equivariant smooth map.

We will use the bicategory of Lie groupoids, principal bibundles, and isomorphisms of bibundles. This requires a bit of work to define. A \define{right action} of a Lie groupoid $H$ on a manifold $P$ consists of
\begin{equation*}
  \text{an \define{anchor} } a:P \to N \text{ and a \define{multiplication} } \mu:P \fiber{a}{t} H \to N,
\end{equation*}
where we denote $p\cdot h := \mu(p,h)$, such that $a(p\cdot h) = s(h)$, and $(p\cdot h)\cdot h' = p\cdot (hh')$ whenever this makes sense, and $p\cdot 1_{a(p)} = p$ for all $p$. We think of the anchor as a ``source'' map determining which elements of $H$ may act on $p$. A \define{(right) principal} $H$\define{-bundle} is a manifold $P$ with right action of $H$, together with an $H$-invariant surjective submersion $\pi:P \to M$ (the bundle map) such that
\begin{equation*}
  P \fiber{a}{t} H \to P \fiber{\pi}{\pi} P, \quad (p,h) \mapsto (p, p\cdot h)
\end{equation*}
is a diffeomorphism. Principal bundles may also be defined with local trivializations, as done in \cite{KM22}. We can similarly define left actions and left principal $G$-bundles.
\begin{definition}
  A \define{bibundle} $P: G \to H$ is a manifold $P$ equipped with a left $G$ action, with anchor $a$, and right $H$ action, with anchor $b$, such that $a$ and $b$ are $H$ and $G$-invariant, respectively. We can use the picture
  \begin{equation*}
    \begin{tikzcd}
      G \circlearrowright & P \ar[dl, "a"'] \ar[dr, "b"]&  \circlearrowleft H \\
      M & & N,
    \end{tikzcd}
  \end{equation*}
  and we may or may not include $G$ and $H$. A bibundle is \define{(right) principal} if $a$ is a right principal $H$-bundle. A bibundle is \define{biprincipal} if both $a$ and $b$ are (right) principal $H$ and (left) principal $G$-bundles. When a biprincipal bibundle $P:G \to H$ exists, we say $G$ and $H$ are \define{Morita equivalent}. 
\end{definition}
\begin{remark}
  \label{rem:1}
  Given a bibundle $P:G \to H$, we have a natural opposite bibundle $P^{\text{op}}:H \to G$. Its left anchor is $b$, right anchor is $a$, and the $H$ and $G$ actions are
  \begin{equation*}
    h \cdot p := p \cdot h^{-1} \text{ and } p \cdot g := g^{-1}\cdot p.
  \end{equation*}
  The bibundle $P$ is right $H$-principal if and only if $P^{\text{op}}$ is left $H$-principal.
\end{remark}

Given two bibundles $P,Q:G \to H$, a \define{bibundle morphism} $\alpha:P \to Q$ is a smooth map that is $G$ and $H$ equivariant. If $P$ and $Q$ are both principal bibundles, then bibundle morphisms are necessarily isomorphisms.

\begin{example}
\label{ex:1}
  To every map of Lie groupoids $\varphi:G \to H$, we can associate a principal bibundle $\langle \varphi \rangle:G \to H$ by taking $\langle \varphi \rangle := M \fiber{\varphi}{t} H$, with anchors $\pr_1$ and $s\pr_2$, and actions
  \begin{equation*}
    g \cdot (x,h) := (t(g), \varphi(g)h) \text{ and } (x,h) \cdot h' := (x,hh').
  \end{equation*}
  Two principal bibundles $\langle \varphi \rangle$ and $\langle \psi \rangle$ are isomorphic if and only if $\varphi$ and $\psi$ are naturally isomorphic. 
\end{example}

We may compose principal bibundles: if $P:G \to H$ and $Q:H \to K$ are principal bibundles, with anchors $a,b$ and $a',b'$, their composition is $Q\circ P := P \fiber{a'}{b} Q/H$, where the $H$ action is $(p,q) \cdot h := (p\cdot h, h^{-1}\cdot q)$. The anchors are
\begin{equation*}
  \alpha([p,q]) := a(p) \text{ and } \beta'([p,q]) := b'(q),
\end{equation*}
and $G$ and $H$ act by multiplication on the factors. However, this composition is not associative; instead, $R \circ (Q\circ P)$ is isomorphic to $(R \circ Q)\circ P$. Thus Lie groupoids, principal bibundles, and bibundle isomorphisms form a bicategory, which we denote \textbf{Bi} (for ``bibundle''). Biprincipal bibundles are the isomorphisms in this category, and the inverse of $P$ is $P^{\text{op}}$.

The functor $\mbf{F}:\mbf{Bi} \to \mbf{Diffeol}$ was described in \cite{Wat22}, and we present it as a proposition.
\begin{proposition}
  \label{prop:10}
  The there is a quotient functor $\mbf{F}:\mathbf{Bi} \to \mathbf{Diffeol}$ which takes $G \rra M$ to $M/G$, is trivial on 2-arrows, and takes a principal bibundle $P:G \to H$ to the unique map $\ol{P}:M/G \to N/H$ making the diagram below commute
  \begin{equation}\label{eq:1}
    \begin{tikzcd}
    & P \ar[dl, "a"'] \ar[dr, "b"] & \\
    M\ar[d, "\pi_G"] & & N \ar[d,"\pi_H"] \\
    M/G \ar[rr, "\ol{P}"] & & N/H.
  \end{tikzcd}
\end{equation}
\end{proposition}

We now introduce the submersions in \textbf{Bi}.
\begin{definition}
  A principal bibundle $P:G \to H$ is a \define{surjective submersion} if its right anchor is a surjective submersion. We also say ``$P$ is a surjective submersive bibundle.''
\end{definition}
If we call $P$ a ``submersion'' when the right anchor is a submersion, the submersions between holonomy groupoids of regular foliations are the submersions between leaf spaces of foliations originally defined by Hilsum and Skandalis \cite{HS87}.
\begin{proposition}
  \label{prop:1}
    To be a surjective submersion is invariant under isomorphism of bibundles. It is also preserved by composition.
  \end{proposition}
  \begin{proof}
    Suppose that $Q,P:G \to H$ are principal bibundles, $P$ is a surjective submersion, and $\alpha:Q \to P$ is a bibundle isomorphism. Then $\alpha$ is a diffeomorphism, and because $b_P$ is a surjective submersion, so is $b_Q = b_P \alpha$.

    Now suppose that $P:G \to H$ and $Q:H \to K$ are surjective submersions, with anchors $a,b$ and $a',b'$ respectively. The composition $Q \circ P$ is the quotient of $P \fiber{}{N} Q$ by the $H$-action $(p,q) \cdot h := (p \cdot h, h^{-1}\cdot q)$. The anchors are
    \begin{equation*}
      \alpha([p,q]) := a(p) \text{ and } \beta'([p,q]) := b'(q),
    \end{equation*}
    and the left $G$ and right $K$-actions are given by multiplication on the left and right factors. We have the commutative diagram ($K$ has base $L$)
    \begin{equation*}
      \begin{tikzcd}
        P\fiber{}{N} Q \ar[r, "\pr_2"] \ar[d, twoheadrightarrow, "\pi"] & Q \ar[d, twoheadrightarrow, "b'"] \\
        Q \circ P \ar[r, "\beta'"] & L,
      \end{tikzcd}
    \end{equation*}
    where the double-headed arrows are surjective submersions. Thus it suffices to show that $\pr_2$ is a surjective submersion. Let $q \in Q$, and choose $(p,q) \in P\fiber{}{N}Q$. We can find $p$ such that $b(p) = a'(q)$ because $b$ is surjective. Since $b$ is a submersion, let $\tau$ be a local section of $b$ taking $a'(q)$ to $p$. Then
    \begin{equation*}
      (Q,q) \dashrightarrow (P \fiber{}{N} Q, (p,q)), \quad q' \mapsto (\tau a'(q'), q')
    \end{equation*}
    is a local section of $\pr_2$ taking $q$ to $(p,q)$. Thus $\pr_2$ is a surjective submersion.
  \end{proof}
  Thus we may form the bicategory $\mbf{Bi}^{\twoheadrightarrow}$, whose objects are Lie groupoids, arrows are surjective submersive bibundles, and 2-arrows are isomorphisms of bibundles.

  \begin{proposition}
    \label{prop:2}
The quotient functor $\mbf{F}:\mbf{Bi} \to \mbf{Diffeol}$ restricts to $\mbf{Bi}^{\twoheadrightarrow} \to \mbf{Diffeol}^{\twoheadrightarrow}$.
\end{proposition}
\begin{proof}
It suffices to show that if $P:G \to H$ is a surjective submersion, then $\ol{P}$ as defined by the diagram \eqref{eq:1} is a surjective local subduction. We recall that \eqref{eq:1} reads
    \begin{equation*}
      \ol{P} \pi_G a = \pi_H b.
    \end{equation*}
    Take $\pi_H(y) \in N/H$, represented by $y \in N$. Say $\ol{P}$ maps $\pi_G(x) \in M/G$ to $\pi_H(y)$, and let $x$ represent $\pi_G(x)$. Fix $p \in a^{-1}(x)$. Then, by the defining equation for $\ol{P}$, both $b(p)$ and $y$ are in the same $H$-orbit.

    Let $\msf{p}:(\msf{U},r) \to (N/H, \pi_H(y))$ be a pointed plot. We must lift $\msf{p}$ to a plot of $M/G$ pointed at $\pi_G(x)$. Since $\pi_H$ is a local subduction mapping $b(p)$ to $\pi_H(y)$, lift $\msf{p}$ along $\pi_H$ to a plot $\msf{q}:(\msf{U},r) \dashrightarrow (N,b(p))$. As $b$ is a surjective submersion, let $\tau$ be a local section of $b$ taking $b(p)$ to $p$. Then, on its domain,
    \begin{equation*}
      (\ol{P} \pi_G a) \tau \msf{q} = \pi_H (b \tau) \msf{q} = \pi_H \msf{q} = \msf{p}.
    \end{equation*}
    Since we also have $\pi_G a \tau q(r) = \pi_G(x)$, the map $\pi_G a \tau \msf{q}$ is the required local lift.
  \end{proof}
  
  We will end this section by discussing differentiable stacks, which we will call simply ``stacks.'' We refer to \cite{Blo08,Ler10,BX11} for technical details on stacks. Stacks form a strict 2-category $\mbf{St}$, and there is an equivalence of bicategories $\mbf{B}:\mbf{Bi} \to \mbf{St}$ (we use the notation from \cite{Ler10}). Thus every stack is isomorphic to one of the form $[M/G] := \mbf{B}G$, and Morita equivalent Lie groupoids represent isomorphic stacks. A principal bibundle $P:G \to H$ gives a map $\mbf{B}P:[M/G] \to [N/H]$, and every map of stacks is isomorphic to one induced by a principal bibundle. In light of Proposition \ref{prop:1}, we can make the following definition:
  \begin{definition}
  A map of differentiable stacks $\mbf{X} \to \mbf{Y}$ is a surjective submersion if it is isomorphic to $\mbf{B}P$ for some surjective submersive bibundle $P:G \to H$, up to the isomorphisms $\mbf{X} \cong [M/G]$ and $\mbf{Y} \cong [N/H]$.
  \end{definition}
   Hoyo and Fernandes \cite[Section 6]{HF19} give a different definition of surjective submersions between stacks, in terms of maps of Lie groupoids. We say a map of Lie groupoids $\varphi:G \to H$ is \define{fully faithful} if
      \begin{equation*}
        \begin{tikzcd}
          G \ar[r, "\varphi"] \ar[d, "{(s,t)}"] & H \ar[d, "{(s,t)}"] \\
          M \times M \ar[r, "\varphi \times \varphi"] & N \times N
        \end{tikzcd}
      \end{equation*}
      is a pullback, \define{essentially surjective} if
      \begin{equation*}
        t\pr_1:H \fiber{s}{\varphi} M \to N
      \end{equation*}
      is a surjective submersion, and \define{Morita} if it is fully faithful and essentially surjective.
      \begin{example}
        \label{ex:2}
        A map is Morita if and only if its associated bibundle (Example \ref{ex:1}) is biprincipal. Many Morita maps come from covers: if $\cl{U}$ is a cover of $M$, the induced functor from the covering groupoid to the trivial groupoid is a Morita map.
      \end{example}

      A \define{weak map} $G \xleftarrow{\varphi} K \xrightarrow{\psi} H$, also denoted $\psi/\varphi:G \to H$, consists of maps $\varphi:K \to G$ and $\psi:K \to H$ such that $\varphi$ is Morita.  Every map of stacks $[M/G] \to [N/H]$ is represented by a weak map, and Hoyo and Fernandes define a surjective submersion to be a map of stacks represented by a weak map with $\psi$ essentially surjective. This notion is invariant under isomorphism of 1-arrows in $\mbf{St}$.

      To compare Hoyo and Fernandes's definition to ours, we note that if $\varphi:G \to H$ is a map of Lie groupoids, the corresponding principal bibundle $\langle \varphi \rangle$ (Example \ref{ex:1}) and the weak map $G \xleftarrow{\id} G \xrightarrow{\varphi} H$ correspond to isomorphic maps of stacks. More generally, a weak map $\psi/\varphi$ and the principal bibundle $\langle \psi \rangle \circ \langle \varphi \rangle^{\text{op}}$ represent isomorphic maps of stacks.
    \begin{proposition}
Suppose the principal bibundle $P:G \to H$ and the weak map $\psi/\varphi:G \to H$ represent isomorphic maps of differentiable stacks. Then $P$ is a surjective submersion if and only if $\psi$ is essentially surjective.
  \end{proposition}

  \begin{proof}
By the discussion above, the map of stacks represented by $\psi/\varphi$ is isomorphic to $\mbf{B}(\langle \psi \rangle \circ \langle \varphi\rangle^{\text{op}})$, which by assumption must be isomorphic to $\mbf{B}P$. Thus $P$ and $\langle \psi \rangle \circ \langle \varphi\rangle^{\text{op}}$ are isomorphic, so by Proposition \ref{prop:1}, $P$ is a surjective submersion if and only if the same is true for $\langle \psi \rangle \circ \langle \varphi\rangle^{\text{op}}$. Since $\langle \varphi \rangle^{\text{op}}$ is a Morita equivalence, it is a surjective submersion. The bibundle $\langle \psi \rangle$ is given by
    \begin{equation*}
      \begin{tikzcd}
        & \langle \psi \rangle := L \fiber{\psi}{t} H \ar[dl, "\pr_1"'] \ar[dr, "s \pr_2"]& \\
        L & & N.
    \end{tikzcd}
    \end{equation*}
The map
    \begin{equation*}
      \langle \psi \rangle \to H \fiber{s}{\psi} L, \quad (z,h) \mapsto (h^{-1},z)
    \end{equation*}
    is a diffeomorphism. Thus, since $s\pr_2(z,h) = t\pr_1(h^{-1},z)$, we see that since $\psi$ is essentially surjective if and only if $\langle \psi \rangle$ is a surjective submersion. By Proposition \ref{prop:1}, we conclude that $\psi$ is essentially surjective if and only if the composition $\langle \psi \rangle\circ \langle \varphi\rangle^{\text{op}}$ is a surjective submersion.
  \end{proof}
  
  Thus our surjective submersions of stacks coincide with those defined by Hoyo and Fernandes. Hoyo and Fernandes \cite{HF19} also explain that their definition is equivalent to Bursztyn, Noseda, and Zhu's \cite{BNZ20}. Therefore, all three notions of surjective submersions of stacks agree.
\section{Lift-complete Lie groupoids}
\label{sec:lift-complete-lie-1}

We want a sub-bicategory of Lie groupoids, whose arrows are surjective submersive bibundles, that is equivalent to $\mbf{QUED}^{\twoheadrightarrow}$. To define this bicategory, we use pseudogroups.

A \define{transition} from a manifold $M$ to a manifold $N$ is a diffeomorphism $f:U \to V$, where $U$ and $V$ are open subsets of $M$ and $N$. We will write $f:M \dashrightarrow N$ to denote a map defined from an open subset of $M$ to an open subset of $N$, and say that $f$ is \define{locally defined}. Recall that $\germ{f}{x}$ denotes the germ of $f$ at $x$.  Given transitions $f:M \dashrightarrow N$ and $g:N \dashrightarrow P$, we understand their composition $g \circ f$ to be defined on $f^{-1}(\dom g)$. This is possibly the empty map.

\begin{definition}
  \label{def:2}
  A \define{pseudogroup} on a manifold $M$ is a collection $\Psi$ of transitions of $M$ such that
  \begin{enumeratea}
  \item $\Psi$ is closed under composition and inversion,
  \item for every open $U \subseteq M$, the restriction $\id|_U:U \to U$ is in $\Psi$, and
  \item if $f:M \dashrightarrow M$ is a transition and, for every $x$ in its domain, $\germ{f}{x}$ has a representative in $\Psi$, then $f$ is in $\Psi$.
  \end{enumeratea}
\end{definition}

Items (a) and (b) together imply that $\Psi$ is closed under restriction to arbitrary open subsets. Items (b) and (c) are sheaf-like properties. The \define{orbit} of $\Psi$ through $x \in M$ is the set $\{\psi(x) \mid \psi \in \Psi\}$. A pseudogroup is \define{countably generated} if there is some countable set of transitions $\Psi_0$ such that every transition in $\Psi$ is locally the composition of transitions in $\Psi_0$. In this case, each orbit is countable, hence totally disconnected.

Take a Lie groupoid $G$. Given an arbitrary section $\sigma$ of $s$, the composition $t \sigma$ is not necessarily a transition. When it is, we call $\sigma$ a \define{local bisection} of $G$. We then associate to $G$ the pseudogroup
\begin{equation*}
  \Psi(G) := \{t\sigma \mid \sigma \text{ is a local bisection of } G\}.
\end{equation*}
\begin{lemma}[{\cite[Proposition 5.3]{MM03}}]
  \label{lem:8}
  For every arrow $g:x \mapsto y$ in $G$, there is some local bisection $\sigma$ such $\sigma(x) = g$. Thus the orbits of $\Psi(G)$ are the orbits of $G$.
\end{lemma}
Conversely, to a pseudogroup $\Psi$ on $M$ we associate the Lie groupoid $\Gamma(\Psi)$, with arrow space consisting of all germs of elements of $\Psi$, base space $M$, source and target
\begin{equation*}
  s(\germ{\psi}{x}) := x \text{ and } t(\germ{\psi}{x}) := \psi(x),
\end{equation*}
and multiplication given by composition. The smooth atlas for $\Gamma(\Psi)$ consists of the charts $x \mapsto \germ{\psi}{x}$, for each $\psi \in \Psi$. The arrow space $\Gamma(\Psi)$ is rarely Hausdorff or second-countable.

\begin{example}
  \label{ex:4}
  For the pseudogroup $\Diff_{\text{loc}}(M)$ consisting of all transitions of $M$, the associated Lie groupoid $\Gamma^M := \Gamma(\Diff_{\text{loc}}(M))$ is the \define{Haefliger groupoid}. When $M$ is connected, the Haefliger groupoid has one orbit, $M$ itself. However, the smooth structure on $M$ given in Remark \ref{rem:5}, for which $t:s^{-1}(x) \to M$ is a principal $(\Gamma^M)_x$-bundle, is that of a 0-dimensional manifold. This is generally the case for the orbit of an arbitrary pseudogroup $\Psi$, when viewed as a $\Gamma(\Psi)$-orbit.
\end{example}
This next remark is an aside, after which we return to the main narrative.
\begin{remark}
  \label{rem:6}
  The situation in Example \ref{ex:4} raises the question whether an orbit of a Lie groupoid (equivalently, of a pseudogroup) admits a manifold structure that is more compatible with the original smooth structure on $M$. For example, we can ask whether the orbit is weakly-embedded: a subset $A \subseteq M$ is a \define{weakly-embedded} submanifold if, with its subset diffeology, it is locally diffeomorphic to Cartesian space, and the inclusion $A \hookrightarrow M$ is an immersion. This notion is intrinsic, as opposed to immersed submanifolds.

  The literature volunteers some partial answers. First, let $\cl{O}$ be an orbit of a Lie groupoid $G$, equipped with its structure as an immersed submanifold. Take a connected component $\cl{O}_0$. By Sussmann's orbit theorem \cite{Sus73}, $\cl{O}_0$ is a maximal integral submanifold of $\rho(\operatorname{Lie}(G))$ (the image of the Lie algebroid of $G$ by its anchor map), and by Stefan's theorem \cite{Ste74} such maximal integral submanifolds are always weakly-embedded. For the Haefliger groupoid, this says that the singletons $\{x\}$ are weakly-embedded.

  Stefan \cite{Ste74} also showed that for a pseudogroup $\Psi$, the orbits of the subgroupoid $\Gamma(\Psi)_0 \leq \Gamma(\Psi)$, whose arrows are germs of transitions that are isotopic to the identity, are weakly-embedded submanifolds. For the Haefliger groupoid with $M$ connected and of dimension at least $2$, Stefan's result equips $M$ with its pre-existing smooth structure. For a proof that the $\Gamma_0^M$-orbit is all of $M$, see \cite{MV94}.

  Using a different approach, Castrigiano and Hayes \cite{CH00} proved that the orbits of a Lie group action are weakly-embedded submanifolds.
\end{remark}

When $G$ is \'{e}tale, we have a functor $\Eff:G \to \Gamma(\Psi(G))$. This is the identity on objects, and takes an arrow $g:x \mapsto y$ to $\germ{t s^{-1}}{x}$, where $s^{-1}$ is the local inverse of $s$ mapping $x$ to $g$. This is the \define{effect} functor, and we denote $\Eff(G) := \Gamma(\Psi(G))$. We call $G$ \define{effective} if $\Eff$ is faithful, in which case $\Eff$ is an isomorphism.

\begin{definition}
  \label{def:5}
  A Lie groupoid $G \rra M$ is \define{lift-complete}\footnote{Thanks to Sergio Zamora Barrera for suggesting the terminology.} if it is \'{e}tale, effective, its orbits are totally disconnected, and
  \begin{enumerate}
  \item[(LC)] whenever $f:U \to M$ is a smooth map such that $\pi f = \pi|_U$, for every $x \in U$ the germ $\germ{f}{x}$ has a representative in $\Psi(G)$.
  \end{enumerate}
  We say a pseudogroup $\Psi$ is \define{lift-complete} if $\Gamma(\Psi)$ is lift-complete. 
\end{definition}
We could consider Lie groupoids satisfying condition \hyperref[def:5]{(LC)} only, but the \'{e}tale, effective, and totally disconnected orbits assumptions will be necessary for all our results, so we bundle them into the definition. Germ groupoids are necessarily \'{e}tale and effective, so a pseudogroup $\Psi$ is lift-complete if and only if its orbits are totally disconnected and condition \hyperref[def:5]{(LC)} holds.
\begin{definition}
  We denote by $\mbf{LiftComp}$ the sub-bicategory of $\mbf{Bi}$ whose objects are lift-complete Lie groupoids. We use $\mbf{LiftComp}^{\twoheadrightarrow}$ to indicate when we only take surjective submersive bibundles for 1-arrows. We let $\mbf{QLiftComp}$ denote the sub-bicategory of $\mbf{LiftComp}$ whose orbit maps are Q-charts.
\end{definition}

\begin{remark}
  \label{rem:7}
  For a Lie groupoid $G$, the orbit map $\pi:M \to M/G$ is a local subduction, and when $G$ is lift-complete, its fibers are totally-disconnected by assumption. The condition \hyperref[def:5]{(LC)} then implies that $\pi$ is quasi-\'{e}tale, thus $M/G$ is quasi-\'{e}tale. It is second-countable because $M$ is second-countable. Therefore the quotient functor $\mbf{F}$ restricts to a functor:
  \begin{itemize}
  \item $\mbf{LiftComp} \to \mbf{QUED}$, as just explained,
  \item $\mbf{F}_{(A)}: \mbf{LiftComp}^{\twoheadrightarrow} \to \mbf{QUED}^{\twoheadrightarrow}$, by Proposition \ref{prop:2}, and
    \item $\mbf{F}_{(B)}:\mbf{QLiftComp} \to \mbf{QMan}$, because quotient maps are surjective. 
   \end{itemize}
\end{remark}
We emphasize that, whereas to be quasi-\'{e}tale is a condition on the orbit map, to be lift-complete is a condition on the groupoid. In particular, even if $\pi:M \to M/G$ is quasi-\'{e}tale, we may have orbit-preserving smooth maps that are not in $\Psi(G)$.

\begin{remark}
  \label{rem:2}
  Given a quotient $\pi:M \to M/{\sim}$ of a manifold, we have the pseudogroup $\Psi(\pi)$ of transitions preserving ${\sim}$. This might be called the gauge pseudogroup of $\pi$. This pseudogroup is not necessarily lift-complete, because generally a smooth map preserving ${\sim}$ need not be a transition. For example, consider the action of $O(2)$ on $\R^2$. The smooth map
  \begin{equation*}
    \R^2\smallsetminus \{(0,0)\} \to \R^2, \quad (x,y) \mapsto (0, \sqrt{x^2+y^2})
  \end{equation*}
  preserves $O(2)$-orbits, but is not a transition, so $\Psi(\pi)$ is not lift-complete here. On the other hand, $\Psi(\pi) = \Psi$ for lift-complete pseudogroups.
\end{remark}

Our two main examples of lift-complete pseudogroups are those induced by certain affine actions on Cartesian space (cf.\ quasifolds), and by certain isometric actions on a Riemannian manifold (cf.\ Riemannian foliations). We have found that these pseudogroups are lift-complete for the same reason, but a sufficient explanation requires a slight detour. We defer it to Subsection \ref{sec:lift-compl-pseud}, and continue here with the question of Morita invariance.

In Subsection \ref{sec:groupoids} we explained that our definition of surjective submersive bibundle gives a notion of a surjective submersion between differentiable stacks. This is because to be a surjective submersion is invariant under isomorphism of principal bibundles. We can also define lift-complete stacks, thanks to the following proposition.

\begin{proposition}
  Suppose $G$ and $H$ are \'{e}tale effective Morita equivalent Lie groupoids.
  \begin{description}
  \item[(A)] If $G$ is lift-complete, so is $H$.
  \item[(B)] If $G$ is an object of $\mbf{QLiftComp}$, so is $H$.
  \end{description}
\end{proposition}
\begin{proof}
  Because $G$ and $H$ are effective, we identify $G$ with $\Eff(G)$ and $H$ with $\Eff(H)$. Fix a biprincipal bibundle $P:G \to H$, with left and right anchors $a$ and $b$. We begin with (A), by verifying property \hyperref[def:5]{(LC)}. Assume $G$ is lift-complete, and let $f:N \dashrightarrow N$ be a smooth map preserving $H$-orbits. We must show that, locally, $f$ coincides with some element of $\Psi(H)$. Fix $y \in \dom f$. Without loss of generality, we assume $f(y) = y$; otherwise, fix an arrow $\germ{h}{y}:y \mapsto f(y)$, and consider $h^{-1}f$.

          Choose $p \in P$ with $y = b(p)$, and set $x := a(p)$. Both $a$ and $b$ are \'{e}tale maps, being bundle projections of principal \'{e}tale groupoid bundles, so we may choose a local inverse $\sigma:(M,x) \dashrightarrow (P,p)$ of $a$ such that $b\sigma :(M,x) \dashrightarrow (N,y)$ is a transition.

          The map $(b\sigma)^{-1} f (b\sigma):(M,x) \dashrightarrow (M,x)$ preserves $G$-orbits. Since $G$ is lift-complete, there is some arrow $\germ{g}{x} \in G$ such that $\germ{(b\sigma)^{-1} f (b\sigma)}{x} = \germ{g}{x}$. Let $(s_G)^{-1}$ be the section of $s_G$ mapping $x$ to $\germ{g}{x}$, and consider the (locally defined) smooth map:
          \begin{table}[h]
            \centering
            \begin{tabular}[h]{c c c c c c c}
              $M$ & $\to$ & $G$ & $\to$ & $P \fiber{a}{a} P$ & $\to$ & $P \fiber{b}{t} H$ \\
              $x'$ & $\mapsto$ & $(s_G)^{-1}(x')$ & $\mapsto$ & $((\sigma(g(x')), (s_G)^{-1}(x') \cdot \sigma(x'))$ & $\mapsto$ & $(\sigma(g(x')), \eta_{x'})$.
            \end{tabular}
          \end{table}
          
          The arrow $\eta_{x'} \in H$ is defined by the condition $(s_G)^{-1}(x') \cdot \sigma(x') = \sigma(k(x')) \cdot \eta_{x'}$. The map $x' \mapsto \eta_{x'}$ is well-defined and smooth because $a:P \to M$ is a right principal $H$-bundle. Observe that
          \begin{align*}
            &s_H(\eta_{x'}) = b(\sigma(k(x')) \cdot \eta_{x'}) = b(s_G^{-1}(x') \cdot \sigma(x')) = b(\sigma(x')) \\
            &t_H(\eta_{x'}) = b(\sigma(g(x'))).
          \end{align*}
          So then the map $(s_H)^{-1}(y') := \eta_{(b\sigma)^{-1}(y')}$ is a section of $s_H$ and
          \begin{equation*}
            t_H((s_H)^{-1}(y')) = (b\sigma) g (b\sigma)^{-1}(y') = f(y').
          \end{equation*}
          This shows that $f$ coincides with the bisection induced by $(s_H)^{-1}(y)$ near $y$. We conclude that $f$ is locally in $\Psi(H)$.
         
          Now we show that the orbits of $H$ are totally disconnected. We do not need to identify $G$ or $H$ with their effects. Suppose that $\msf{p}:\msf{U} \to N$ is a plot with image contained in an $H$-orbit. Take $r \in \msf{U}$, and set $y := \msf{p}(r)$. Let $\tau$ be a section of $b$ defined near $y$, and set $p := b(y)$, and $x := a(p)$. Then $a \tau \msf{p}:(\msf{U},r) \dashrightarrow (M,x)$ is a pointed plot with image in the $G$-orbit $G \cdot x$. Indeed, if $r'$ is near $r$, take an arrow $h:\msf{p}(r') \to y$, which exists because $\msf{p}$ has image in an $H$-orbit. Then
          \begin{equation*}
            b\tau \msf{p}(r') = \msf{p}(r') = b(p\cdot h),
          \end{equation*}
          so $\tau p(r')$ and $p\cdot h$ are in the same $b$-fiber. By principality of the $G$-action, we can find $g \in G$ with $\tau p(r') = g \cdot p \cdot h$. Then
          \begin{equation*}
            s(g) = a(p\cdot h) = a(p) \text{ and } t(g) = a(g\cdot p \cdot h) = a\tau \msf{p}(r').
          \end{equation*}
          Since $G \cdot x$ is totally disconnected, the plot $a\tau \msf{p}$ must be constant near $r$. But $a\tau$ is a transition, so it must be $\msf{p}$ that is constant near $r$. Since $r$ was arbitrary, we conclude that $\msf{p}$ is locally constant on $\msf{U}$, and thus the $H$-orbits are totally disconnected.

          Finally, we verify (B) by checking that, if $\pi_G$ is a Q-chart, so is $\pi_H$. Let $\msf{p},\ti{\msf{p}}:\msf{U} \to N$ be plots with $\pi_H \msf{p} = \pi_H \ti{\msf{p}}$. Take $r \in \msf{U}$, and suppose that $y := \msf{p}(r) = \ti{\msf{p}}(r)$. We need $\germ{\msf{p}}{r} = \germ{\ti{\msf{p}}}{r}$. Let $\tau$ be a section of $b$ near $y$, such that $a\tau$ is a transition. Then, by definition of $\ol{P}$ (Proposition \ref{prop:10}),
          \begin{equation*}
            (\pi_G a \tau) \msf{p} = (\ol{P})^{-1} \pi_H \msf{p} = ((\ol{P})^{-1} \pi_H) \ti{\msf{p}} = \pi_G a \tau \ti{\msf{p}}.
          \end{equation*}
          Since $\pi_G$ is a Q-chart, and $a\tau \msf{p}(r) = a\tau \ti{\msf{p}}(r)$, we have $\germ{a\tau \msf{p}}{r} = \germ{a\tau \ti{\msf{p}}}{r}$. But $a\tau$ is a transition, so $\germ{\msf{p}}{r} = \germ{\ti{\msf{p}}}{r}$ as required.
        \end{proof}

        \begin{definition}
          A differentiable stack is \define{lift-complete} if it is represented by some lift-complete Lie groupoid.
        \end{definition}
        As it should be, this definition is independent of the choice of representative \'{e}tale Lie groupoid.

        \subsection{Necessary lemmas}
        \label{sec:necessary-lemmas}
        
The following three lemmas are crucial to the sequel. The first concerns lifts of local subductions. It is also in \cite[Proposition 3.10]{Vill23}, but we give the proof here for completeness.

        \begin{lemma}
          \label{lem:1}
          Suppose $\pi_X:M \to X$ and $\pi_Y:N \to Y$ are quasi-\'{e}tale maps, with $M$ and $N$ manifolds, and $f:X \to Y$ is a smooth map.
          \begin{enumeratea}
          \item For every $x \in M$ and $y \in N$ such that $f(\pi_X(x)) = \pi_Y(y)$, there is some local lift $\varphi$ of $f$ taking $x$ to $y$.
          \item If $f$ is surjective, for every $y \in N$ there is some local lift mapping to $y$.
          \item If $f$ is a local subduction, any lift is a submersion.
          \item If $f$ is a local diffeomorphism, any lift is a local diffeomorphism.
          \end{enumeratea}
        \end{lemma}
        \begin{proof}
          We prove (a) and (b) first. Fix $y \in N$ with $\pi(y) \in f(X)$. Choose $x \in M$ such that $f(\pi_X(x)) = \pi_Y(y)$. Since $\pi_Y$ is a local subduction, we may locally lift $f \pi_X$ along $\pi_Y$ to a pointed map $\varphi:(M,x) \dashrightarrow (N,y)$. This proves (a) and (b).

          Now, for (c), assume that $f$ is a local subduction. Since $f\pi_X$ is a local subduction, we may locally lift $\pi_Y$ along $f\pi_X$ to a pointed map $s: (N,y) \dashrightarrow (M,x)$. Observe that
          \begin{equation*}
            (\pi_Y \varphi) s = (f \pi_X) s = \pi_Y,
  \end{equation*}
  so $\varphi s$ is a smooth map that preserves $\pi_Y$ fibers. Since $\pi_Y$ is quasi-\'{e}tale, this means $\varphi s$ is a local diffeomorphism. It follows that $\varphi$ is a submersion.

  For (d), if $f$ is a local diffeomorphism, then choose a local lift $s$ of $f^{-1} \pi_Y$ along $\pi_X$, where $f^{-1}$ denotes the local inverse of $f$ mapping $y$ to $x$. Then
  \begin{equation*}
   (\pi_X s) \varphi = f^{-1} (\pi_Y \varphi) = (f^{-1} f) \pi_X = \pi_X, 
 \end{equation*}
 so $s \varphi$ is a smooth map that preserves $\pi_X$ fibers. Since $\pi_X$ is quasi-\'{e}tale, this means that $s\varphi$ is a local diffeomorphism. It follows that $\varphi$ is an immersion, and we know it is a submersion from (c), hence it is a local diffeomorphism.
\end{proof}
The next lemma is about the local structure of submersions between manifolds. Here we use the \define{functional diffeology} on $C^\infty(V,U)$, for manifolds $V$ and $U$. A map
\begin{equation*}
  \msf{V} \to C^\infty(V,U), \quad v \mapsto \sigma^v
\end{equation*}
is a plot of the functional diffeology if the evaluation map
\begin{equation*}
  \msf{V} \times V \to U, \quad (v,x') \mapsto \sigma^v(x')
\end{equation*}
is smooth.

\begin{lemma}
  \label{lem:2}
  Suppose $\varphi:M \to N$ is a submersion. For every $x \in M$, there are neighbourhoods $U$ of $x$ and $V$ of $y:=\varphi(x)$, and a plot
  \begin{equation*}
    \R^{m-n} \supseteq \msf{V} \to C^\infty(V, U), \quad v \mapsto \sigma^v,
  \end{equation*}
  such that
  \begin{itemize}
  \item $\msf{V}$ is a connected neighbourhood of $0$ in $\R^{m-n}$;
  \item $\sigma^v:V \to U$ is a section of $\varphi$ for all $v$;
    \item for each $x' \in U$, there is some $v \in \msf{V}$ such that $\sigma^v(\varphi(x')) = x'$.
    \end{itemize}
    Furthermore, we may impose that $\sigma^0(y) = x$.
  \end{lemma}

  \begin{proof}
      Because $\varphi$ is a submersion, we can put it in normal form. Specifically, we may fix charts
  \begin{align*}
    r_M:\msf{U} \times \msf{V} \to M, \quad r_N:\msf{U} \to N
  \end{align*}
  of $M$ and $N$ centred at $x$ and $\varphi(x)$, respectively, such that:
  \begin{itemize}
  \item $\varphi(r_M(\msf{U} \times \msf{V})) = r_N(\msf{U})$,
  \item $\msf{V}$ is connected, and
  \item $(r_N)^{-1}\varphi r_M: \msf{U} \times \msf{V} \to \msf{U}$ is the projection $\pr_1$.
  \end{itemize}
Set
  \begin{equation*}
    U := r_M(\msf{U} \times \msf{V}), \text{ and } V := r_N(\msf{U}).
  \end{equation*}
For each $v \in \msf{V}$, define the smooth map
  \begin{equation*}
    \sigma^v:V \to U, \quad y' \mapsto r_M((r_N)^{-1}(y'),v).
  \end{equation*}
  In a commutative diagram,
  \begin{equation*}
    \begin{tikzcd}[sep = large]
      \msf{U} \times \msf{V} \ar[r, "r_M", "\cong"'] \ar[d, "\pr_1"] & U \ar[d, "\varphi"] \\
      \msf{U} \ar[r, "r_N", "\cong"'] \ar[u, bend left, "{u \mapsto (u,v)}"] & V. \ar[u, bend left, "\sigma^v"]
    \end{tikzcd}
  \end{equation*}
  We claim that
  \begin{equation*}
    \sigma: \msf{V} \to C^\infty(V, U), \quad v \mapsto \sigma^v
  \end{equation*}
  is a plot of $C^\infty(V,U)$ with the desired properties. It is a plot because the evaluation map
  \begin{equation*}
    \msf{V} \times V \to U, \quad (v,x') \mapsto \sigma^v(x') = r_M((r_N)^{-1}(x'),v)
  \end{equation*}
   is smooth. As for the required properties, first, we chose $\msf{V}$ to be a connected neighbourhood of $0$ in $\R^{m-n}$. Second, the $\sigma_v$ are sections of $\varphi$:
  \begin{equation*}
    \varphi \sigma^v(x') = \varphi(r_M((r_N)^{-1}(x'),v)) = r_N \pr_1((r_N)^{-1}(x'),v) = x'.
  \end{equation*}
  Third, for each $x' \in U$, if we write $x' = r(u,v)$, then $v$ is the required element of $\msf{V}$, since
  \begin{equation*}
    \sigma^v(\varphi(r_M(u,v))) = \sigma^v(r_N(u)) = r_M(u,v) = x'.
  \end{equation*}
  Finally, as we chose centred charts, $x = r(0,0)$ and $\varphi(x) = r_N(0)$, so $\sigma^0(\varphi(x)) = x$.
  \end{proof}
This refined description of submersions allows the last lemma for this section, concerning the possible lifts of a surjective local subduction between quotient spaces.
\begin{lemma}
  \label{lem:3}
  Let $H$ be a lift-complete Lie groupoid, and suppose $\varphi, \ti{\varphi}:M \to N$ both descend to the same map to $N/H$. If
  \begin{description}
  \item[(A)] $\varphi$ is a submersion, or
  \item[(B)] $\pi:N \to N/H$ is a Q-chart (no assumption on $\varphi$),
  \end{description}
  then for each $x \in M$, there is some transition $T^x := T(\ti{\varphi}, \varphi, x)$ in $\Psi(H)$ such that $\germ{T^x\varphi}{x} = \germ{\ti{\varphi}}{x}$, and this condition uniquely determines $\germ{T^x}{\varphi(x)}$.
  Furthermore, the map
  \begin{equation}
    \label{eq:4}
    M \dashrightarrow \Eff(H), \quad x' \mapsto \germ{T^{x'}}{\varphi(x')} 
  \end{equation}
is smooth.
\end{lemma}
\begin{proof}
For (A), assume that $\varphi$ is a submersion. Fix $x \in M$, and take the smooth family of sections $\{\sigma^v:V \to U\}_{v \in \msf{V}}$ from Lemma \ref{lem:2}. The map $\ti{\varphi} \sigma^v$ preserves $H$-orbits:
  \begin{equation*}
    \pi \ti{\varphi} \sigma^v = \pi \varphi \sigma^v = \pi.
  \end{equation*}
  Now, for each $y' \in V$, consider the set
  \begin{equation*}
    \{\ti{\varphi}\sigma^v(y') \mid v \in \msf{V}\} \subseteq N.
  \end{equation*}
  As the $\ti{\varphi}\sigma^v$ preserve $H$-orbits, this is a subset of the $H$-orbit of $x'$. By our assumption that $H$ is lift-complete, this set is totally disconnected. On the other hand, this set is also the image of the smooth map
  \begin{equation*}
    \msf{V} \to N, \quad v \mapsto \ti{\varphi}\sigma^v(y').
  \end{equation*}
  We chose $\msf{V}$ to be connected, so this image must be connected, and since it is also totally disconnected, it is a singleton. In other words, $\ti{\varphi}\sigma^v(y')$ does not depend on $v$.

 We define $T^x := \ti{\varphi}\sigma^0$. By the above discussion, this is also $\ti{\varphi}\sigma^v$ for any $v \in \msf{V}$. For arbitrary $x' = r(u,v) \in U$, we have
  \begin{equation*}
    T^x \varphi(x') = \ti{\varphi}\sigma^0\varphi(x') = \ti{\varphi}\sigma^v\varphi(x') =\ti{\varphi}(x'),
  \end{equation*}
  so $T^x\varphi$ and $\ti{\varphi}$ agree on $U$. Since $H$ is lift-complete, and $T^x$ preserves $H$-orbits, we may shrink its domain so that $T^x \in \Psi(H)$. This proves the first assertion.

  Now we check that $\germ{T^x}{\varphi(x)}$ is uniquely determined. Suppose that $h$ is a transition of $N$ and $\germ{h\varphi}{x} = \germ{\ti{\varphi}}{x}$. Then $T^x\varphi$ and $h\varphi$ agree in a neighbourhood of $x$. But $\varphi$ is a submersion, hence an open map, so $T^x$ and $h$ agree on a neighbourhood of $\varphi(x)$. In other words, $\germ{T^x}{\varphi(x)} = \germ{h}{\varphi(x)}$.

  This lets us show that \eqref{eq:4} is smooth. Fix $x \in M$, and a neighbourhood $U$ such that $T^x\varphi|_U = \ti{\varphi}|_U$. Then for any $x' \in U$, we must have
  \begin{equation*}
    \germ{T^x\varphi}{x'} = \germ{\ti{\varphi}}{x'} = \germ{T^{x'}\varphi}{x'},
  \end{equation*}
and the previous paragraph gives $\germ{T^x}{\varphi(x')} = \germ{T^{x'}}{\varphi(x')}$. So the map in \eqref{eq:4} is, near $x$, $x' \mapsto \germ{T^x}{\varphi(x')}$, which is smooth.
  
Now we prove (B). Take $\varphi$ arbitrary, but assume that $\pi:N \to N/H$ is a Q-chart. Fix $x \in M$. Since $\pi$ is a local subduction, and $\pi(\varphi(x)) = \pi(\ti{\varphi}(x))$, lift $\pi$ along itself to a pointed map $T^x:(N, \varphi(x)) \dashrightarrow (N, \ti{\varphi}(x))$. This preserves $H$-orbits, and thus restricts to a transition in $\Psi(H)$ by the assumption that $H$ is lift-complete.  Both $T^x\varphi$ and $\ti{\varphi}$ agree at $x$, and project to the same map along $\pi$, so by condition \hyperref[def:4]{(Q)}, $\germ{T^x\varphi}{x} = \germ{\ti{\varphi}}{x}$. Therefore, we have found $T^x$.

To check $\germ{T^x}{\varphi(x)}$ is uniquely determined, suppose that $h$ is a transition of $N$ and $\germ{h\varphi}{x} = \germ{\ti{\varphi}}{x}$. Then $h^{-1}T$ is a smooth map that preserves $H$-orbits and fixes $\varphi(x)$, so by \hyperref[def:4]{(Q)} its germ at $\varphi(x)$ must coincide with $\germ{\id_N}{\varphi(x)}$. Similarly, $\germ{T^xh^{-1}}{\ti{\varphi}(x)} = \germ{\id_N}{\ti{\varphi}(x)}$. This proves that $\germ{T^x}{\varphi(x)} = \germ{h}{\varphi(x)}$. Finally, the map \eqref{eq:4} is smooth by the same reasoning as in (A).
\end{proof}

\subsection{Lift-complete pseudogroups and PDEs}
\label{sec:lift-compl-pseud}

In this subsection, we show that pseudogroups generated by countable subgroups of
\begin{itemize}
\item affine transformations of a Cartesian space, or
\item isometries of a Riemannian manifold
\end{itemize}
are lift-complete. Despite the \emph{prima facie} difference between these to examples, they are lift-complete for the same reason. This reason is best stated in the language of jets and partial differential equations (PDEs). Our main source if \cite{Olv93}.

Given a smooth submersion $p:P \to M$, we declare two sections defined near $x \in M$ to have the same $k$-jet at $x$ if their $k$-th order Taylor polynomials agree at $x$ in some, which implies any, local coordinates. We denote the equivalence class of $\sigma$ under this relation by $j_x^k\sigma$. The collection of all jets
\begin{equation*}
  J^kP := \{j_x^k\sigma \mid \sigma \text{ is a section of }p, \ x \in \dom \sigma\}
\end{equation*}
admits a canonical structure of a smooth finite-dimensional manifold, for which the projection $s: j_x^k\sigma \mapsto x$ is a submersion. A \define{PDE of order} $k$ is a subset $R \subseteq J^kP$, with $s(R) = M$. A \define{(local) solution} to a PDE is a local section $\sigma$ of $p$ for which $j^k_x\sigma \in R$ for each $x \in \dom \sigma$. If $P := M \times M \to M$ is the trivial bundle, we identify local sections of $p$ with locally defined smooth maps on $M$. Thus, given a pseudogroup $\Psi$ on $M$, we may form the subset $J^k\Psi \subseteq J^k(M \times M)$ of $k$-jets of elements of $\Psi$. This is not a submanifold in general.

In pursuit of our examples, we require some definitions.

\begin{definition}
  \label{def:6}
  A pseudogroup $\Psi$ is \define{complete} if, for every $x$ and $y$ in $M$, there are open neighbourhoods $V_x$ of $x$ and $V_y$ of $y$ such that, whenever $\psi \in \Psi$ maps some $x' \in V_x$ into $V_y$, there is some $\psi_{\ext} \in \Psi$ defined on all of $V_x$ for which $\germ{\sigma}{x'} = \germ{\sigma_{\ext}}{x'}$.
\end{definition}
This is the definition used by Haefliger, for instance see \cite{Hae85}. Beware that $\psi_\ext$ is not necessarily an extension of $\psi$. It is, however, an extension of $\psi|_U$ for some neighbourhood $U$ of $x'$, and this motivates our notation. We could similarly define a ``complete'' PDE, but this is not necessary. Any pseudogroup generated by a group of diffeomorphisms is complete.

We introduce the following class of PDEs.
\begin{definition}
  \label{def:7}
We call a PDE $R \subseteq J^kP$ \define{quasi-analytic} if, for all solutions $\sigma$ and $\sigma'$, and all $x \in M$, whenever $\germ{\sigma}{x} = \germ{\sigma'}{x}$, we have $\sigma = \sigma'$ on the connected component of $\dom \sigma \cap \dom \sigma'$ containing $x$.
\end{definition}
If $R = J^k\Psi$ is quasi-analytic, then we call $\Psi$ quasi-analytic. Pseudogroups of affine transformations of a Cartesian space, and pseudogroups of local isometries of a Riemannian manifold, are quasi-analytic.
\begin{remark}
An \emph{a priori} different definition of quasi-analytic pseudogroups (cf.\ \cite{AC09}, who attribute the definition to Haefliger) asks that for all $\psi \in \Psi$ and $x \in \dom \psi$, whenever there is some open subset $U \subseteq \dom \psi$ such that $\psi|_U = \id|_U$, and $x \in \ol{U}$ (where the closure is taken relative to $\dom \psi$), there exists some open neighbourhood $U' \subseteq \dom \psi$ of $x$ such that $\psi|_{U'} = \id|_{U'}$. Using point-set topology, one may show that this definition agrees with ours.
\end{remark}

\begin{example}
Quasi-analytic pseudogroups arise from $(G,X)$ structures. These begin with a manifold $X$, and a Lie group $G$ acting on $X$ ``analytically,'' meaning that whenever $\germ{g}{x} = \germ{g'}{x}$ for some $x \in X$, we have $g = g'$. A $(G,X)$-structure on a manifold $M$ is an atlas of $X$-valued charts whose transition functions lie in the pseudogroup generated by $G$. The pseudogroup generated by $G$, or generated by the transition functions, is quasi-analytic.
\end{example}

We may now state sufficient conditions for a pseudogroup to be lift-complete.
\begin{proposition}
  \label{prop:11}
  Let $\Psi$ be a complete and countably-generated pseudogroup on $M$. Suppose that for some $k \geq 1$, $\Psi$ consists of solutions to a $k$-th order quasi-analytic PDE $R$, such that $R$ is a closed subset of $J^k(M \times M)$. Then $\Psi$ is lift-complete.
\end{proposition}

\begin{proof}
The orbits of $\Psi$ are countable subsets of $M$, hence totally disconnected. It only remains to verify condition \hyperref[def:5]{(LC)}. Take a smooth map $f:U \to M$ that preserves the $\Psi$ orbits, and let $x_0 \in U$. Fix a countable collection $\{\psi_i\} \subseteq \Psi$ such that $\psi_i(x_0) = y_0$ and, for each $\psi \in \Psi$, there is some $i$ with $\germ{\psi}{x_0} = \germ{\psi_i}{x_0}$. Such a collection exists because $\Psi$ is countably-generated.

     Apply the assumption that $\Psi$ is complete to the points $x_0$ and $y_0 := f(x_0)$, to yield the distinguished neighbourhoods $V_{x_0}$ and $V_{y_0}$. We may shrink $V_{x_0}$ so that it is a connected subset of $U$ such that $f(V_{x_0}) \subseteq V_{y_0}$. We restrict the domain of $f$ to $V_{x_0}$, and to keep notation simple, we assume $U = V_{x_0}$. Completeness of $\Psi$ gives extensions $(\psi_i)_{\ext}:U \to M$ of (the germs of) $\psi_i$ based at $x_0$.

     Let
      \begin{equation*}
        \Delta_i := \{x \in U \mid f(x) = (\psi_i)_{\ext}(x)\}.
      \end{equation*}
       This is a closed subset of $U$, because $(\psi_i)_{\ext}$ are defined on $U$. There are countably many $\Delta_i$. We claim that $\bigcup \Delta_i = U$. The inclusion ``$\subseteq$'' is clear, so assume that $x \in U$. Since $f$ preserves $\ti{\Psi}$-orbits, there is some $\psi \in \ti{\Psi}$ with $f(x) = \psi(x) \in V_{y_0}$. By completeness, we have an extension $\psi_{\ext}:U \to M$ of (the germ of) $\psi$ based at $x$. By choice of the collection $\{\psi_i\}$, we have $\germ{\psi_{\ext}}{x_0} = \germ{(\psi_i)_{\ext}}{x_0}$ for some $i$. But both $\psi_{\ext}$ and $(\psi_i)_{\ext}$ are solutions to $R$, which is quasi-analytic, so $\psi_\ext = (\psi_i)_\ext$ on the connected component of $x_0$ in $\dom \psi_{\ext} \cap \dom (\psi_i)_{\ext}$. This intersection is simply $U$, which is connected. Therefore $x \in \Delta_i$, and $U \subseteq \bigcup_i \Delta_i$. 

      By the Baire category theorem, $V := \bigcup \Delta_i^\circ$ is an open dense subset of $U$.  Now consider the map
      \begin{equation*}
        j^kf:U \to J^k(M \times M), \quad x \mapsto j^k_xf.
      \end{equation*}
      This is smooth, and maps each $\Delta_i^\circ$ into $R$. Thus $j^kf(V) \subseteq R$. But $V$ is dense in $U$, and we assume that $R$ is closed in $J^k(M \times M)$, so therefore
      \begin{equation*}
       R \supset \ol{j^kf(V)} = \ol{j^kf(\ol{V})} = \ol{j^kf(U)} \supset j^kf(U).
     \end{equation*}
Thus $f$ is a solution to $R$. By density of $V$, there is some $i$ and some point $x$ for which $\germ{f}{x} = \germ{(\psi_i)_{\ext}}{x}$. But $R$ is quasi-analytic, so the solutions $f$ and $(\psi_i)_{\ext}$ must agree on the component of $\dom f \cap \dom (\psi_i)_{\ext}$ about $x$. As this intersection is all of $U$, which is connected, we conclude that $\germ{f}{x_0} = \germ{\psi_i}{x_0}$.  Since $x_0$ was arbitrary, we have shown that $\ti{\Psi}$ is lift-complete. 
\end{proof}

\begin{corollary}
  \label{cor:1}
  Suppose that $\Psi$ is a complete and countably-generated sub-pseudogroup of either 
  \begin{itemize}
  \item $\Aff(\R^n)$, the group of affine transformations of $\R^n$, or
  \item $\Iso(M,g)$, the group of isometries of a Riemannian manifold $(M,g)$.
  \end{itemize}
  Then $\Psi$ is lift-complete.
\end{corollary}
\begin{proof}
  In either case, $\Psi$ is complete and countably generated. To apply Proposition \ref{prop:11}, we require the PDE $R$.
  \begin{itemize}
  \item For the case of affine transformations, take $k = 2$ and
    \begin{equation*}
      R := \{j_x^k\sigma \mid \sigma(x) = Ax+b \text{ for } A \in \operatorname{Mat}(n,\R), \ b \in \R^n\}.
    \end{equation*}
    In other words, $R$ is the PDE requiring the second derivative of $\sigma$ to vanish identically. This is a closed condition, so $R$ is closed in $J^2(M \times M)$. The PDE $R$ is also quasi-analytic, because affine transformations are analytic functions of $\R^n$.
  \item For the case of isometries, take $k = 1$ and $R := J^1\operatorname{Iso}(M,g)$. This is closed in $J^1(M \times M)$, because to preserve the metric is a closed condition on the 1-jet. It is also quasi-analytic, since local isometries are determined by their germs at a point.
  \end{itemize}
\end{proof}

\begin{example}
  \label{ex:3}
  A diffeological quasifold, defined in \cite{IZP20,KM22} is a second-countable diffeological space that is locally diffeomorphic to the quotients $\R^n/\Gamma$, where $n$ is fixed and $\Gamma$ is a countable subgroup of $\Aff(\R^n)$ that may vary. By Corollary \ref{cor:1}, each $\R^n/\Gamma$ is quasi-\'{e}tale, so diffeological quasifolds are quasi-\'{e}tale. It was shown in both \cite{IZP20} and \cite{KM22} that quasifolds are quasi-\'{e}tale, and the arguments in Corollary \ref{cor:1} are similar to both of those proofs.
\end{example}

More generally, the effective quasifold groupoids from \cite{KM22} are lift-complete. A quasifold groupoid is a Lie groupoid $G$ with Hausdorff arrow space, such that for every $x \in M$, there is some isomorphism of Lie groupoids $\varphi:(\Gamma \ltimes \R^n)|_{\msf{V}} \to G|_U$, where $U$ is a neighbourhood of $x$, $\msf{V}$ is an open subset of $\R^n$, and $\Gamma$ is a countable group acting affinely on $\R^n$.

\begin{proposition}
  \label{prop:3}
  Effective quasifold groupoids are lift-complete.
\end{proposition}
\begin{proof}
This is similar to \cite[Lemma 5.2]{KM22}. Fix an effective quasifold groupoid $G$. We first check condition \hyperref[def:5]{(LC)}. Let $f:U \to M$ be a smooth function preserving $G$ orbits. Because $G$ is effective, we can identify $G$ with $\Eff(G)$. Let $x \in U$, and assume that $f(x) = x$; otherwise, take an arrow $\germ{g}{x}:x \mapsto f(x)$ and consider $g^{-1}f$.  Shrinking $U$ if necessary, we may assume that it is the domain of some quasifold groupoid chart $\varphi:(\Gamma \ltimes \R^n)|_{\msf{V}} \to G|_U$, and that $f:U \to U$. The map $\varphi^{-1} f \varphi$ is a transition of $\msf{V}$, and it preserves $\Gamma$-orbits. Therefore, by Corollary \ref{cor:1} there is some $\gamma \in \Gamma$ such that $\germ{\varphi^{-1} f \varphi }{\varphi^{-1}(x)} = \germ{\gamma}{\varphi^{-1}(x)}$. Then the map
  \begin{equation*}
    \sigma:U \dashrightarrow G|_U, \quad x' \mapsto \varphi(\gamma, \varphi^{-1}(x'))
  \end{equation*}
  is a section of $s$, and $t\sigma = f$. So then $f$, restricted to a neighbourhood of $x$, is in $\Psi(G)$. Since $x$ was arbitrary, we conclude that $f$ is locally in $\Psi(G)$, as desired.

  Now we show that $G$ has totally disconnected orbits. Let $p:\msf{U} \to M$ be a plot with image in a $G$-orbit. Take $r \in \msf{U}$ and denote $x := p(r)$. Let $\varphi:(\Gamma \ltimes \R^n)|_{\msf{V}} \to G|_U$ be a quasifold groupoid chart near $x$. Then $\varphi^{-1} p:\msf{U} \dashrightarrow \msf{V}$ maps into a $\Gamma$-orbit of $\R^n$, which is a countable set because $\Gamma$ is countable. This means $\varphi^{-1} p$ is locally constant near $r$, and $\varphi^{-1}$ is a diffeomorphism, so $p$ is locally constant near $r$. Since $r$ was arbitrary, we conclude that $p$ is locally constant, hence the $G$-orbits are totally disconnected. 
\end{proof}

In Section \ref{sec:exampl-riem-foli}, we will use Corollary \ref{cor:1} to prove that the holonomy pseudogroup of a Riemannian foliation is lift-complete, and thus its leaf space is quasi-\'{e}tale. 

\section{An equivalence of categories}
\label{sec:an-equiv-categ}
Here prove our main result, Theorem \ref{thm:1}. First, we address essential surjectivity.

\begin{proposition}
  \label{prop:4}
  Suppose that $X$ is a second-countable quasi-\'{e}tale diffeological space.
  \begin{description}
  \item[(A)] There is some lift-complete Lie groupoid $\Gamma(X) \rra \cl{M}$ such that $\cl{M}/\Gamma(X)$ is diffeomorphic to $X$.
  \item[(B)]  If $X$ is also a Q-manifold, then the quotient $\cl{M} \to \cl{M} / \Gamma(X)$ is a Q-chart.
  \end{description}

\end{proposition}
\begin{proof}
  We begin with (A), after which (B) is immediate. Let $\{\pi_i:M_i \to X\}$ be a collection of quasi-\'{e}tale maps into $X$, whose images cover $X$. Let $\cl{M} := \bigsqcup_i M_i$. We assume that $X$ is second-countable, so we can assume our collection of charts is countable, and thus $\cl{M}$ is a manifold. Let
  \begin{equation*}
    \Psi := \{\text{transitions }\varphi_{ji}:M_i \dashrightarrow M_j \mid \pi_j \circ \varphi_{ji} = \pi_i|_{\dom \varphi}\}.
  \end{equation*}
  This is a pseudogroup on $\cl{M}$, and we set $\Gamma(X) := \Gamma(\Psi) \rra \cl{M}$. Being a germ groupoid, it is \'{e}tale and effective. We verify condition \hyperref[def:5]{(LC)}. Let $f:U \to \cl{M}$ be a smooth map that preserves $\Psi(\Gamma(X)) = \Psi$-orbits. We must show $f$ is locally in $\Psi$. Fix $x \in U$, set $y := f(x)$, and assume that $x \in M_i$ and $y \in M_j$. Since $f$ is smooth, we may take neighbourhoods $U_i$ of $x$ and $U_j$ of $y$ such that $f$ restricts to a map $f:(U_i,x) \to (U_j,y)$.

  Fix $\varphi_{ji}:U_i \dashrightarrow U_j$ from $\Psi$ such that $y = \varphi_{ji}(x)$, which is possible because $f$ preserves $\Psi$ orbits. Take $T(f, \varphi_{ji}, x) \in \Psi$ as given by Lemma \hyperref[lem:3]{\ref{lem:3} (A)}. Then $\germ{f}{x} = \germ{T(f, \varphi_{ji}, x)^{-1}\varphi_{ji}}{x}$, and thus $f$ is locally in $\Psi$.

  Now we show that the fibers of $\pi:\cl{M} \to \cl{M}/\Gamma(X)$ are totally disconnected. Suppose $p:\msf{U} \to \cl{M}$ is a plot with image in a $\Psi$-orbit. Let $r \in \msf{U}$. We may shrink $\msf{U}$ so that $p$ restricts to a plot $p: \msf{U} \to M_i$. Then the image of $p$ is contained in the fiber of $\pi_i$, which is totally disconnected because $\pi_i$ is quasi-\'{e}tale. Therefore $p$ is constant near $r$, and since $r$ was arbitrary, $p$ is locally constant on $\msf{U}$. This means the $\Psi$-orbits are totally disconnected.

  Finally, we give a diffeomorphism $f:\cl{M}/\Gamma(X) \to X$. This will be the unique map $f$ completing the diagram
  \begin{equation}
    \label{eq:5}
    \begin{tikzcd}
      & \cl{M} \ar[dl, "\pi"'] \ar[dr, "\bigsqcup_i \pi_i"] & \\
      \cl{M}/\Gamma(X) \ar[rr, "f"] & & X.
    \end{tikzcd}
  \end{equation}
  Such a map, if it exists, is a surjective smooth local subduction because both $\pi$ and $\bigsqcup_i \pi_i$ are surjective local subductions. If $f$ is also injective, then it is a diffeomorphism (its inverse is smooth for the same reason $f$ is smooth). We must define $f$ by
  \begin{equation*}
    f(\pi(x)) := \pi_i(x) \text{ where } x \in M_i.
  \end{equation*}
 To see this is well-defined, suppose that $\pi(x) = \pi(x')$, and $x' \in M_j$. Then there is a transition $\varphi_{ji}:(M_i, x) \dashrightarrow (M_j, x')$ in $\Psi$. By the definition of $\Psi$, we have
  \begin{equation*}
    f(\pi(x')) = \pi_j(x') = \pi_j\varphi_{ji}(x) = \pi_i(x) = f(\pi(x)),
  \end{equation*}
  so $f$ is well-defined. To see $f$ is injective, suppose that $f(\pi(x)) = f(\pi(x'))$. Then representatives $x \in M_i$ and $x' \in M_j$ satisfy $\pi_i(x) = \pi_j(x')$. Applying Lemma \ref{lem:1} to $\id:X \to X$ and the quasi-\'{e}tale maps $\pi_i$ and $\pi_j$, we get some transition (after restricting its domain) $\varphi_{ji}:(M_i,x) \dashrightarrow (M_j, x')$ with $\pi_j \varphi_{ji} = \pi_i$. Thus $x$ and $x'$ are in the same $\Psi$-orbit, and $\pi(x) = \pi(x')$.

  To prove (B), take the $\pi_i:M_i \to X$ to be Q-charts, and note that $\bigsqcup_i \pi_i$ is a Q-chart, hence so is $\pi = f^{-1} \circ \bigsqcup_i \pi_i$.
\end{proof}

Now, we show $\mbf{F}$ is is full.
      \begin{proposition}
\label{prop:5}
Suppose $G$ and $H$ are lift-complete Lie groupoids.
\begin{description}
\item[(A)] A local subduction $f:M/G \to N/H$ lifts to a principal bibundle $P:G \to H$ whose right anchor is a submersion. If $f$ also surjective, then $P$ is a surjective submersion.
  \item[(B)] If the quotient maps $\pi_G:M \to M/G$ and $\pi_H:N \to N/H$ are Q-charts, any smooth map $f:M/G \to N/H$ lifts to a principal bibundle $P:G \to H$.
\end{description}

      \end{proposition}
      \begin{proof}
        For now, assume neither (A) nor (B), but only that we have a smooth map $f:M/G \to N/H$. Because $G$ and $H$ are effective, and the effect isomorphism descends to the identity on the orbit spaces, it suffices to give a bibundle $P:\Eff(G) \to \Eff(H)$. For simplicity, we will denote $\Eff(G)$ also by $G$ and $\Eff(H)$ by $H$. It will be convenient to give a bibundle $P:H \to G$, and then take the opposite bundle. As a set,
        \begin{equation*}
          P := \{\germ{\varphi}{x} \mid \varphi \text{ locally lifts } f\}.
        \end{equation*}
        For each local lift $\varphi$, we have the bijection
        \begin{equation*}
          \{\germ{\varphi}{x} \mid x \in \dom \varphi\} \to \dom \varphi, \quad \germ{\varphi}{x} \mapsto x,
        \end{equation*}
        and these are the charts for the manifold structure of $P$. The anchor maps $b:P \to N$ and $a:P \to M$ are given by
        \begin{equation*}
          \begin{tikzcd}
            & (P,\germ{\varphi}{x}) \ar[dl, "b"'] \ar[dr, "a"] & \\
            (N, \varphi(x)) & & (M, x).
          \end{tikzcd}
        \end{equation*}
        These are smooth by the choice of smooth structure on $P$. The actions are given by composition,
        \begin{equation*}
          \germ{h}{\varphi(x)} \cdot \germ{\varphi}{x} := \germ{h\varphi}{x} \text{ and } \germ{\varphi}{x} \cdot \germ{g}{g^{-1}(x)} := \germ{\varphi g}{g^{-1}(x)}.
        \end{equation*}
        It is straightforward to check that these are smooth. Evidently $a$ is $H$-invariant and $b$ is $G$-invariant. The map $a$ is surjective by Lemma \hyperref[lem:1]{\ref{lem:1} (a)}.

  If we assume that $f$ is a submersion, then the map $b$ is a submersion because, by Lemma \hyperref[lem:1]{\ref{lem:1} (c)}, the lifts of $f$ are submersions. If $f$ is surjective, Lemma \hyperref[lem:1]{\ref{lem:1} (b)} implies that the images of all possible liftings of $f$ cover $N$, so $b$ is surjective. To complete the proof for cases (A) and (B), all that remains is to show that $a:P \to M$ is a left principal $H$-bundle. We must show the action map $(\germ{h}{\varphi(x)}, \germ{\varphi}{x}) \mapsto (\germ{h\varphi}{x}, \germ{\varphi}{x})$ is a diffeomorphism. We claim its inverse is:
        \begin{equation*}
         \Phi: P \fiber{a}{a} P \to H \fiber{s}{b} P, \quad (\germ{\ti{\varphi}}{x}, \germ{\varphi}{x}) \mapsto (\germ{T(\ti{\varphi}, \varphi, x)}{\varphi(x)}, \germ{\varphi}{x}).
        \end{equation*}
        If $f$ is a submersion (i.e.\ (A) holds), then because $H$ is lift-complete and the lifts $\ti{\varphi}$ and $\varphi$ are submersions, we can apply Lemma \hyperref[lem:3]{\ref{lem:3} (A)} to obtain $T^x := T(\ti{\varphi}, \varphi, x) \in \Psi(H)$. If the quotient maps are Q-charts (i.e.\ (B) holds), then we can apply Lemma \hyperref[lem:3]{\ref{lem:3} (B)} to obtain $T^x$. In either case, we have $T^x$ subject to the conclusions of Lemma \ref{lem:3}.

        The germ $\germ{T^x}{\varphi(x)}$ is in $\Eff(H)$ by definition of the effect functor, and our identification of $\Eff(H)$ with $H$. A quick review of Lemma \ref{lem:3} shows that the germ of $T^x$ at $\varphi(x)$ depends only on the germs of $\ti{\varphi}$ and $\varphi$ at $x$, so $\Phi$ is well-defined. It is smooth because, if we fix $\ti{\varphi}$ and $\varphi$, the map $x \mapsto \germ{T^x}{\varphi(x)}$ is smooth by Lemma \ref{lem:3}, and $x \mapsto \germ{\varphi}{x}$ is smooth because it is a chart of $P$. The map $\Phi$ is an inverse of the action map by the definition of $T^x$, but we write out the verification for completeness.
        \begin{align*}
          &(\germ{\ti{\varphi}}{x}, \germ{\varphi}{x}) \mapsto (\germ{T^x}{\varphi(x)}, \germ{\varphi}{x}) \mapsto (\germ{T^x\varphi}{x}, \germ{\varphi}{x}) = (\germ{\ti{\varphi}}{x}, \germ{\varphi}{x}) \\
          &(\germ{h}{\varphi(x)}, \germ{\varphi}{x}) \mapsto (\germ{h\varphi}{x}, \germ{\varphi}{x}) \mapsto (\germ{T(h\varphi, \varphi, x)}{\varphi(x)}, \germ{\varphi}{x}) = (\germ{h}{\varphi(x)}, \germ{\varphi}{x}).
        \end{align*}
        The desired bibundle is the opposite of $P$.

      \end{proof}

      \begin{corollary}
        \label{cor:4}
        In the setting of Proposition \ref{prop:5}, if $f$ is a diffeomorphism, it lifts to a Morita equivalence.
      \end{corollary}
      
      \begin{proof}
        Assume $f$ is a diffeomorphism. We take all notation from the proof of Proposition \ref{prop:5}. By Lemma \hyperref[lem:1]{\ref{lem:1} (d)}, the local lifts $\varphi$ of $f$ are local diffeomorphisms. To show $b:P \to N$ is a right principal $G$-bundle, we assert that the inverse of the action map $(\germ{\varphi}{x}, \germ{g}{g^{-1}(x)}) \mapsto (\germ{\varphi}{x}, \germ{\varphi g}{g^{-1}(x)})$ is
        \begin{equation*}
          P \fiber{b}{b} P \to P \fiber{a}{t} P, \quad (\germ{\varphi}{x}, \germ{\ti{\varphi}}{x'}) \mapsto (\germ{\varphi}{x}, \germ{T(\ti{\varphi}^{-1}, \varphi^{-1}, \varphi(x))^{-1}}{x}).
        \end{equation*}
        The notation $\varphi^{-1}$ denotes the local inverse of $\varphi$ that maps $x$ to $\varphi(x)$, and similarly for $\ti{\varphi}^{-1}$. Verifying that we have indeed given the inverse is similar to the argument in Proposition \ref{prop:5}, so we do not repeat it.
      \end{proof}
      \begin{remark}
        Furthermore, we could show that if $f$ is a local diffeomorphism, then it lifts to a locally invertible bibundle, in the sense of \cite{KM22}.
      \end{remark}

      Finally, we show $\mbf{F}$ is faithful up to isomorphism of 1-arrows. We first do this for bibundles induced by maps of Lie groupoids. Note that while we used $\varphi$ to denote a map of Lie groupoids in Section \ref{sec:groupoids}, we have already used that symbol for other purposes in this section, so here we denote Lie groupoid maps by $X$ and $Y$ instead. These do not refer to diffeological spaces.
\begin{lemma}
\label{lem:4}
Suppose $G$ and $H$ are lift-complete Lie groupoids, and that $X,Y:G \to H$ are smooth functors (maps of Lie groupoids) that induce the same map on the orbit spaces.
\begin{description}
\item[(A)] If this induced map is a surjective local subduction, then $X$ and $Y$ are naturally isomorphic.
\item[(B)] If $G$ and $H$ are objects of $\mbf{QLiftComp}$, then $X$ and $Y$ are naturally isomorphic (without any assumption on the induced map).
\end{description}
 
\end{lemma}
\begin{proof}
  Because we assume the groupoids are effective, we may identify $G$ and $H$ with the germ groupoids $\Eff(G)$ and $\Eff(H)$, respectively. We propose to define
  \begin{equation*}
   \alpha:M \to H \quad x \mapsto \germ{T(Y,X,x)}{X(x)}.
 \end{equation*}
 If $X$ and $Y$ descend to a surjective local subduction, then $X,Y:M \to N$ are submersions by Lemma \hyperref[lem:1]{\ref{lem:1} (b)}, and we use Lemma \hyperref[lem:3]{\ref{lem:3} (A)} to get $T(Y,X,x)$. If $\pi_G$ and $\pi_H$ are Q-charts, we apply Lemma \hyperref[lem:3]{\ref{lem:3} (B)} to get $T(Y,X,x)$. In either case, we have $T(Y,X,x)$ subject to the conclusions of Lemma \ref{lem:3}. In particular, $\alpha$ is smooth. 

 We show $\alpha:X \implies Y$ is a natural transformation. First, for an arrow $\germ{g}{x}$ of $G$, whose effect transformation we denote by $g$, we have
    \begin{equation}\label{eq:3}
      X(\germ{g}{x}) = \germ{T(Xg,X,x)}{X(x)} \text{ and } Y(\germ{g}{x}) = \germ{T(Yg,Y,x)}{Y(x)}.
    \end{equation}
To see this, fix $x \in M$. Say $X(\germ{g}{x}) = \germ{h}{X(x)}$, for some transition $h \in \Psi(H)$. Then $X$, being continuous, must map points $\germ{g}{x'}$ near $\germ{g}{x}$ into the open subset $\{\germ{h}{y} \mid y \in \dom h\}$ of $H$. In other words, $X(\germ{g}{x'}) = \germ{h}{X(x')}$. Since $X$ is a functor, this means that $hX(x') = Xg(x')$ near $x$, and so $h = T(Xg, X, x)$ near $x$. The case for $Y$ is identical.
    
It follows that, for any arrow $\germ{g}{x}$ in $G$, the following commutes
    \begin{equation*}
      \begin{tikzcd}[sep=large]
        X(x) \ar[r, "\alpha(x)"] \ar[d, "{X(\germ{g}{x})}"] & Y(x) \ar[d, "Y({\germ{g}{x})}"] \\
        X(g(x)) \ar[r, "\alpha(g(x))"] & Y(g(x)).
      \end{tikzcd}
    \end{equation*}
    First, it is clear from the definition of $T(\cdot, \cdot, \cdot)$ that
    \begin{equation*}
      Y(\germ{g}{x}) \cdot \alpha(x) = \germ{T(Yg, Y, x)}{Y(x)} \cdot \germ{T(Y,X,x)}{X(x)} = \germ{T(Yg, X, x)}{X(x)}.
    \end{equation*}
On the other hand,
    \begin{equation*}
      \alpha(g(x)) \cdot X(\germ{g}{x}) = \germ{T(Y,X,g(x))}{X(g(x))} \cdot \germ{T(Xg, X, x)}{X(x)},
    \end{equation*}
    and $\germ{T(Y,X,g(x))}{X(g(x))} = \germ{T(Yg,X,x)}{X(g(x))}$ because $T(Y,X,g(x'))X(g(x')) = Y(g(x'))$ for $x'$ near $x$, thus
    \begin{equation*}
      \alpha(g(x)) \cdot X(\germ{g}{x}) = \germ{T(Yg,X,x)}{X(x)}.
    \end{equation*}    
\end{proof}

This immediately extends to bibundles.
\begin{proposition}
  \label{prop:6}
  Suppose $G$ and $H$ are lift-complete Lie groupoids, and $P,Q:G \to H$ are bibundles that induce the same map on the orbit spaces.
  \begin{description}
  \item[(A)] If this induced map is a surjective local subduction, then $P$ and $Q$ are isomorphic.
  \item[(B)] If $\pi_G$ and $\pi_H$ are Q-charts, then $P$ and $Q$ are isomorphic (without any assumption on the induced map).
  \end{description}

\end{proposition}
\begin{proof}
  By \cite[Lemma 3.37]{Ler10}, there is a cover $\iota: \cl{U} \to M$ of $M$ such that both $P \circ \langle \iota \rangle$ and $Q \circ \langle \iota \rangle$ are induced by maps $X_P, X_Q:\cl{U} \to H$ (writing $\cl{U}$ for the covering groupoid). Then $X_P$ and $X_Q$ induce the same map $|P| = |Q|$ on orbit spaces, so these functors are naturally isomorphic by Lemma \hyperref[lem:4]{\ref{lem:4} (A)} or \hyperref[lem:4]{\ref{lem:4} (B)}. But $\langle \iota \rangle$ is a Morita equivalence, thus $P$ and $Q$ must be isomorphic.
\end{proof}

Now we compile the proof of Theorem \ref{thm:1}, which we restate here for convenience.
\begin{theorem*}
  The quotient functor $\mbf{F}$ restricts in two ways:
  \begin{description}
  \item[(A)] To $\mbf{F}_{(A)}:\mbf{LiftComp}^{\twoheadrightarrow} \to \mbf{QUED}^{\twoheadrightarrow}$. This restriction is:
  \begin{itemize}
  \item essentially surjective on objects;
  \item full on 1-arrows;
  \item faithful up to isomorphism of 1-arrows.
  \end{itemize}
  Furthermore, the stack represented by a lift-complete Lie groupoid is determined by its diffeological orbit space, and $\mbf{F}_{(A)}$ gives an equivalence of categories if we descend to $\mbf{HS}$.
  \item[(B)] To $\mbf{F}_{(B)}:\mbf{QLiftComp} \to \mbf{QMan}$, satisfying the same conclusions as $\mbf{F}_{(A)}$.
  \end{description}
\end{theorem*}
\begin{proof}
  The fact that $\mbf{F}_{(A)}$ restricts to a functor $\mbf{LiftComp}^{\twoheadrightarrow} \to \mbf{QUED}^{\twoheadrightarrow}$ is given by Proposition \ref{prop:2} and Remark \ref{rem:7}. That $\mbf{F}_{(A)}$ is essentially surjective, full on 1-arrows, and faithful up to isomorphism of 1-arrows, is Proposition \hyperref[prop:4]{\ref{prop:4} (A)}, \hyperref[prop:5]{\ref{prop:5} (A)}, and \hyperref[prop:6]{\ref{prop:6} (A)}, respectively. Corollary \ref{cor:4} implies that if two lift-complete Lie groupoids have diffeomorphic orbit spaces, then they are Morita equivalent, i.e.\ they represent isomorphic stacks. Finally, in $\mbf{HS}$ the objects are Lie groupoids and the arrows are isomorphism classes of principal bibundles. Thus, upon descending to $\mbf{HS}$, the quality ``faithful up to isomorphism of 1-arrows'' becomes simply ``faithful,'' and we see $\mbf{F}_{(A)}$ gives an equivalence of categories. The argument for $\mbf{F}_{(B)}$ is the same, except we appeal to Propositions \hyperref[prop:4]{\ref{prop:4} (B)}, \hyperref[prop:5]{\ref{prop:5} (B)}, and \hyperref[prop:6]{\ref{prop:6} (B)} instead.
\end{proof}

\subsection{Examples and extensions}
\label{sec:examples-extensions}

Theorem \hyperref[thm:1]{\ref{thm:1} (A)} fails if we try and allow for all arrows, as seen in the next two examples from \cite{IZKZ10}. We refer there for details. In both cases, the relevant orbit spaces are orbifolds.

\begin{example}[{\cite[Example 24]{IZKZ10}}]
  \label{ex:5}
  Let $\rho_n:\R \to [0,1]$ be a bump function with support inside $\left[\frac{1}{n+1}, \frac{1}{n}\right]$, and take $\sigma := (\sigma_1, \sigma_2, \ldots) \in \{-1,1\}^{\mathbb{N}}$, and set
  \begin{equation*}
    f_\sigma(x) :=
    \begin{cases}
      \sigma_ne^{-1/x}\rho_n(x) &\text{if } 1/(n+1) < x \leq 1/n \\
      0 &\text{if } x > 1 \text{ or } x \leq 0.
    \end{cases}
  \end{equation*}
  Then all the $f_\sigma$ descend to the same map $\R \to \R/\Z_2$. However, no two of the associated Lie groupoid maps $\R \to \R \rtimes \Z_2$ are naturally isomorphic. This shows that, in general, $\mbf{F}_{(A)}$ is not faithful up to isomorphism of $1$-arrows.
\end{example}

\begin{example}[{\cite[Example 25]{IZKZ10}}]
  \label{ex:6}
  Take $\rho_n$ as above, set $r := \sqrt{x^2+y^2}$, and let $f:\R^2 \to \R^2$ be defined by
  \begin{equation*}
    f(x,y) :=
    \begin{cases}
      e^{-r}\rho_n(r)(r,0) &\text{if } 1/(n+1) < r \leq 1/n \text{ and } n \text{ is even} \\
      e^{-r}\rho_n(r)(x,y) &\text{if } 1/(n+1) < r \leq 1/n \text{ and } n \text{ is odd} \\
      0 &\text{if } r > 1 \text{ or } r = 0.
    \end{cases}
  \end{equation*}
  For each $m \geq 2$, the function $f$ descends to a smooth function $\ol{f}: \R^2/Z_m \to \R^2/\Z_m$. However, $f$, does not upgrade to a functor $\R^2 \rtimes \Z_m \to \R^2 \rtimes \Z_m$, and in fact there is no such functor (equivalently, bibundle) that induces $\ol{f}$. This shows that, in general, $\mbf{F}_{(A)}$ is not full on 1-arrows.
\end{example}

It is also not possible to add arrows in $\mbf{LiftComp}^{\twoheadrightarrow}$ without the essential image of $\mbf{F}$ leaving $\mbf{QUED}^{\twoheadrightarrow}$.

\begin{proposition}
  \label{prop:8}
  Suppose that $G$ and $H$ are a Lie groupoids, $H$ is lift-complete, and $P:G \to H$ is a bibundle whose induced map on orbit spaces $\ol{P}:M/G \to N/H$ is a surjective local subduction. Then $P$ is a surjective submersion. 
\end{proposition}
\begin{proof}
  Denote the anchor maps of $P:G \to H$ by $a$ and $b$. Let $y \in N$. All of $a$, $\pi_H$, and $\ol{P}$ are surjective local subductions, thus so is their composition. Choose $p' \in P$ such that $\ol{P}\pi_Ga(p') = \pi_H(y)$. By the defining property of $\ol{P}$, we have $\pi_H(b(p')) = \pi_H(y)$. Take an arrow $h:y \mapsto b(p')$. Then $b(p' \cdot h) = s(h) = y$, so $b$ is surjective.

Take any $p$ with $b(p) = y$. Since $\ol{P}\pi_Ga$ is a local subduction, lift $\pi_H$ along $\ol{P}\pi_Ga$ to a pointed map $\tau:(N,y) \dashrightarrow (P,p)$. By definition of $\ol{P}$ and choice of $\tau$, we have, on its domain,
  \begin{equation*}
    (\pi_H b)\tau = \ol{P}\pi_G a \tau = \pi_H.
  \end{equation*}
  Since $H$ is lift-complete, $b\tau$ must, locally, be a transition, thus $b$ must be a submersion.
\end{proof}

This allows a pleasant statement about quotient maps. If $H$ is a lift-complete Lie groupoid, then by Theorem \hyperref[thm:1]{\ref{thm:1} (A)} there is a surjective submersive bibundle lifting the quotient $N \to N/H$, and it is isomorphic to every other such surjective submersion. By Proposition \ref{prop:8}, every lift of the quotient is a surjective submersion, thus there is a unique lift of the quotient up to isomorphism. In this case, we have a convenient representative $P:(N \rra N) \to H$, given by $P:= H$, anchors $t$ and $s$, trivial left $(N \rra N)$-action, and right $H$-action by multiplication of arrows. This bibundle should be viewed as a trivial principal bundle, in the same way $G \to \{*\}$ is a trivial principal bundle for a Lie group.

We will conclude by discussing an application to gluing bibundles. If $P:G \to H$ is a principal bibundle, and $U \subseteq M$ is an open subset, we can form the \define{restriction} $P|_U:G|_U \to H$, which is also a principal bibundle. The following example from Lerman \cite{Ler10} illustrates some awkward behaviour of restrictions in $\mbf{HS}$.
\begin{example}[{\cite[Lemma 3.41]{Ler10}}]\label{ex:7}
  A principal bibundle $P:(S^1\rra S^1) \to (\Z_2 \rra \{*\})$ is entirely determined by a right principal $\Z_2$-bundle over $S^1$. There are two distinct isomorphism classes of such bundles. However, all such bundles are equivalent when restricted to any contractible open subset of $S^1$. Thus we have two distinct morphisms in $\mbf{HS}$ which agree on an open cover of $S^1$.  
\end{example}
This situation is not covered by Theorem \ref{thm:1}, because the groupoid $\Z_2 \rra \{*\}$ is not effective. In fact, we can avoid the pitfalls of Example \ref{ex:7} entirely under suitable conditions.

\begin{proposition} 
  Take Lie groupoids $G$ and $H$, and cover $\cl{U}$ of $M$.
  \label{prop:9}
  \begin{description}
  \item[(A)] Suppose that $G$ and $H$ are lift-complete Lie groupoids.
    \begin{enumeratea}
  \item If $P,Q:G \to H$ are  surjective submersive bibundles, and $P|_U \cong Q|_U$ for each $U \in \cl{U}$, then $P \cong Q$.
  \item If, for each $U_\alpha \in \cl{U}$, we have a surjective submersive bibundle $P_\alpha:G|_{U_\alpha} \to H$, and isomorphisms $P_\alpha|_{U_{\alpha\beta}} \cong P_\beta|_{U_{\alpha\beta}}$ for all $\alpha,\beta$, then there is some surjective submersive bibundle $P:G \to H$ such that $P|_{U_\alpha} \cong P_\alpha$.
     \end{enumeratea}
  \item[(B)] Suppose that $G$ and $H$ are objects in $\mbf{QLiftComp}$. Then (a) and (b) hold without requiring any bibundle to be a surjective submersion.
  \end{description}

\end{proposition}
\begin{proof}
 We work in case (A).
  \begin{enumeratea}
    \item The induced maps $\ol{P|_U}$ and $\ol{Q|_U}$ from $\pi_G(U) \to N/H$ are identical, since isomorphic bibundles induce the same map on orbit spaces. The images $\pi_G(U)$ cover $M/G$, so the maps $\ol{P}$ and $\ol{Q}$ must be identical. But then, by Theorem \hyperref[thm:1]{\ref{thm:1} (A)}, $P$ and $Q$ must be isomorphic.
    \item Set $\ol{U_\alpha} := \pi_G(U_{\alpha})$. Then we have a collection of surjective submersions $\ol{P_\alpha}:\ol{U_\alpha} \to N/H$ which agree on the intersections $\ol{U_{\alpha\beta}}$. Therefore there is some smooth function $f:M/G \to N/H$ such that $f|_{\ol{U_\alpha}} = \ol{P_\alpha}$. This function is a surjective local subduction since its restriction to each open subset is a surjective local subduction. Therefore it lifts to a surjective submersion $P:G \to H$. The restriction $P|_{U_\alpha}$, which is also a surjective submersive bibundle, induces the same map induced by $P_\alpha$. By Theorem \hyperref[thm:1]{\ref{thm:1} (A)}, we conclude that $P|_{U_\alpha} \cong P_\alpha$.
    \end{enumeratea}
    The case (B) is the same, except we appeal to Theorem \hyperref[thm:1]{\ref{thm:1} (B)}.
\end{proof}

\section{Application to Riemannian foliations}
\label{sec:exampl-riem-foli}
In this last section, we show how Riemannian foliations give examples of lift-complete Lie groupoids and quasi-\'{e}tale diffeological spaces, to which we can apply Theorem \ref{thm:1}. First, we review Riemannian foliations, with details from \cite{Mol88, MM03}.

A codimension-$q$ \define{foliation} of a manifold $M$ is a partition $\cl{F}$ of $M$ into connected weakly-embedded (see Remark \ref{rem:6}) submanifolds $L$ of codimension $q$, called \define{leaves}, such that the associated distribution $x \mapsto T_xL \in T_x\cl{F}$ is smooth. In this case, the distribution $T\cl{F}$ is involutive.\footnote{Conversely, by the Jacobi-Clebsch-Deahna-Frobenius theorem, every involutive distribution induces a unique foliation.} The data of a foliation is equivalent to a \define{Haefliger cocycle}. This is a countable collection of submersions $\{s_i:U_i \to \R^q\}$, such that there exist diffeomorphisms $\gamma_{ji}:s_i(U_{ij}) \to s_j(U_{ij})$ (writing $U_{ij} := U_i \cap U_j$) for which $\gamma_{ji} s_i = s_j$ on $U_{ij}$.  The connected components of the fibers of the $s_i$, called \define{plaques}, glue together into leaves of a foliation. Conversely, every foliation admits a Haefliger cocycle.

Fix a Haefliger cocycle. For convenience, assume that each $s_i$ admits a global section $\sigma_i$. Let $\lambda:x \mapsto y$ be a path in a leaf $L$. Cover $\lambda([0,1])$ with a chain $U_0, \ldots, U_k$ such that $U_{i(i+1)} \neq \emptyset$, and $x \in U_0$ and $y \in U_k$. Choose a global section $\sigma_0$ of $s_0$ through $x$, and let $S := S_0$ be the immersed submanifold $\sigma_0(s_0(U_0))$. Similarly get $T:= S_k$. Then $x'$ near $x$ is in a plaque of $U_0$, which intersects a unique plaque of $U_1$, which intersects a unique plaque of $U_2$, etc. Ultimately, we find that $x'$ determines a unique plaque of $U_k$, which meets $T$ at a unique point $y'$. The transition $x' \mapsto y'$ is called the \define{holonomy transition} associated to $\lambda$, denoted $\hol^{T,S}(\lambda):S \dashrightarrow T$.

One can show that for fixed $T$ and $S$, the holonomy $\hol^{T, S}(\lambda)$ does not depend on the $U_i$, or even the homotopy class of $\sigma$. If we choose $S_i$ for every $U_i$, and let $\cl{S} := \bigsqcup_i S_i$, then $\cl{S}$ is a \define{complete transversal} to $\cl{F}$ (meaning the map $\cl{S} \hookrightarrow M$ is an immersion that is transverse to $\cl{F}$ and meets every leaf). The set of all holonomy transitions $\cl{S} \dashrightarrow \cl{S}$ is the \define{holonomy pseudogroup} $\Psi_{\cl{S}}(\cl{F})$ of $\cl{F}$.

If we chose other transversals $T'$ and $S'$, we have
\begin{equation*}
  \hol^{T', S'}(\lambda) = \hol^{T', S}(\ol{y}) \circ \hol^{T, S}(\lambda) \circ \hol^{S, S'}(\ol{x}),
\end{equation*}
where $\ol{y}$ and $\ol{x}$ are the constant paths. This means that the pseudogroup $\Psi_{\cl{S}'}(\cl{F})$ is equivalent\footnote{In the sense of pseudogroups. Equivalently, their germ groupoids are Morita equivalent.} to $\Psi_{\cl{S}}(\cl{F})$, so we call either of these ``the'' holonomy pseudogroup. More generally, we can represent the holonomy pseudogroup on any complete transversal $\cl{S}$. 

The holonomy pseudogroup captures the diffeology of the leaf space.
\begin{lemma}
  \label{lem:5}
  The leaf space $M/\cl{F}$ is diffeomorphic to $\cl{S}/\Psi_{\cl{S}}(\cl{F})$.
\end{lemma}
\begin{proof}
  We will only sketch the proof, and have given a complete argument (but along slightly different lines) in \cite{Miy22}. The diffeomorphism is given by
  \begin{equation*}
    M/\cl{F} \to \cl{S}/\Psi_{\cl{S}}(\cl{F}), \quad L \mapsto [x] \text{ where } x \in L \cap \cl{S}.
  \end{equation*}
  This is well-defined because any two points in $L \cap \cl{S}$ are connected by some leafwise path, hence by some holonomy transition. Its inverse is the map $[x] \mapsto L_x$, which is well-defined because points joined by some holonomy transition must be in the same leaf. Smoothness is a consequence of the fact that bijection fits into the diagram
  \begin{equation*}
    \begin{tikzcd}
      & \cl{S} \ar[dl, "\pi_{\cl{F}} \circ \iota"'] \ar[dr, "\pi_{\Psi_{\cl{S}}(\cl{F})}"] & \\
      M/\cl{F} \ar[rr] & & \cl{S}/\Psi_{\cl{S}}(\cl{F}),
    \end{tikzcd}
  \end{equation*}
  and the downward arrows are surjective subductions.
\end{proof}

A \define{Riemannian foliation} is a foliation equipped with a non-negative and symmetric (0,2) tensor $g$ on $M$ such that $\iota_X g = 0$ and $\cl{L}_X g = 0$ for all vector fields $X$ tangent to $\cl{F}$. We call $g$ a \define{transversely Riemannian metric}. If $\cl{S}$ is a complete transversal to a Riemannian foliation $\cl{F}$, then $g$ descends to a Riemannian metric on $\cl{S}$, and the holonomy pseudogroup consists of local isometries. Furthermore,
\begin{lemma}
  \label{lem:7}
  The holonomy pseudogroup of a Riemannian foliation is complete.
\end{lemma}
\begin{proof}
  See \cite[Appendix D, Proposition 2.6]{Mol88}.
\end{proof}
Lemma \ref{lem:7} holds even if the induced metric on $\cl{S}$ is not complete. We may now show that leaf spaces of Riemannian foliations are quasi-\'{e}tale.

\begin{proposition}
\label{prop:7}
  The leaf space of a Riemannian foliation is an object of $\mbf{QUED}$.
\end{proposition}
\begin{proof}
Let $(M, \cl{F}, g)$ be a Riemannian foliation. Take a complete transversal $\cl{S}$, and form $\Psi_{\cl{S}}(\cl{F})$. Recall that $g$ descends to a Riemannian metric on $\cl{S}$ for which $\Psi_{\cl{S}}(\cl{F})$ consists of local isometries. Furthermore, $\Psi_{\cl{S}}(\cl{F})$ is countably generated, because we can choose a countable Haefliger cocycle, and it is complete by Lemma \ref{lem:7}. Therefore, by Corollary \ref{cor:1}, the holonomy pseudogroup $\Psi_{\cl{S}}(\cl{F})$ is lift-complete, and in particular $\cl{S}/\Psi_{\cl{S}}(\cl{F})$ is quasi-\'{e}tale. By Lemma \ref{lem:5}, this means $M/\cl{F}$ is quasi-\'{e}tale. The leaf space is second-countable because $M$ is second-countable.
\end{proof}

To each foliation $\cl{F}$ we can also associate a Lie groupoid $\Hol(\cl{F}) \rra M$, called the \define{holonomy groupoid}. Its arrow space consists of leafwise paths $\lambda$ modulo the following equivalence: $\lambda \sim \ti{\lambda}$ if $\hol^{S,S}(\lambda^{-1} * \ti{\lambda})$ is the identity. We denote the holonomy class of $\lambda$ by $[\lambda]$. A class $[\lambda]$ is an arrow $[\lambda] x \mapsto y$, and multiplication is given by concatenation. For the smooth structure on $\Hol(\cl{F})$, see \cite[Proposition 5.6]{MM03}. The orbits of $\Hol(\cl{F})$ are the leaves of $\cl{F}$, with their usual diffeology (cf.\ the discussion in Remark \ref{rem:6}).

If $\cl{S} \hookrightarrow M$ is a complete transversal, we can form the pullback groupoid $\Hol_{\cl{S}}(\cl{F}) \rra \cl{S}$, whose arrows are holonomy classes of paths with endpoints in $\cl{S}$. This is the \define{\'{e}tale holonomy groupoid}. It is a Lie groupoid because $\cl{S} \hookrightarrow M$ is is transverse to $\cl{F}$, and it is Morita equivalent to $\Hol(\cl{F})$ because the induced functor is a Morita map. One can show that $\Hol_{\cl{S}}(\cl{F})$ is \'{e}tale and effective, and $\Psi(\Hol_{\cl{S}}(\cl{F})) = \Psi(\cl{S})$. We can then prove Corollary \ref{cor:2} from the Introduction, which we restate here.

\begin{corollary*}
    Two Riemannian foliations have diffeomorphic leaf spaces if and only if their holonomy groupoids are Morita equivalent.
\end{corollary*}
\begin{proof}
Given a Riemannian foliation $(M, \cl{F}, g)$, its holonomy groupoid $\Hol(\cl{F})$ is Morita equivalent to $\Hol_{\cl{S}}(\cl{F})$ for any complete transversal $\cl{S}$. This \'{e}tale holonomy groupoid is lift complete by Proposition \ref{prop:7} and the fact $\Psi(\Hol_{\cl{S}}(\cl{F})) = \Psi(\cl{S})$. Therefore we can apply Theorem \ref{thm:1}, from which this corollary is immediate.
\end{proof}

This is not true for arbitrary foliations.
\begin{example}
  \label{ex:9}
This is a summary of results from \cite[Section 7]{KM22}. Take a smooth bounded function $h:\R \to \R$ that is flat at 0 and positive everywhere else. Let $\varphi$ be the time 1 flow of the vector field $h\partial_x$, and let
  \begin{equation*}
    \ti{\varphi}(x) := 
    \begin{cases}
      \varphi(x) &\text{if } x \geq 0 \\
      \varphi^{-1}(x) &\text{if } x < 0.
    \end{cases}
  \end{equation*}
By iterating $\varphi$ and $\tilde{\varphi}$, we get two $\Z$ actions on $\R$ with the same orbit spaces $X := \R/\varphi = \R/\ti{\varphi}$.\footnote{The space $X$ is an example of a quotient of $\R$ by a countable group that is not a quasifold.}  However, the corresponding action groupoids are not Morita equivalent. By suspending the actions, we obtain two foliations whose orbit spaces are diffeomorphic to $X$, and whose \'{e}tale holonomy groupoids are isomorphic to the corresponding action groupoids. Therefore these foliations do not satisfy Corollary \ref{cor:2}.
\end{example}

We finish by discussing transverse equivalence of Riemannian foliations.
\begin{definition}
  \label{def:3}
  Two foliations $(M,\cl{F}_M)$ and $(N, \cl{F}_N)$ are \define{Molino transversely equivalent} if there is a manifold $P$ and two surjective submersions with connected fibers $a:P \to M$ and $b:P \to N$ such that $a^{-1}(\cl{F}) = b^{-1}(\cl{F})$. Here the leaves of $a^{-1}(\cl{F})$ are $a^{-1}(L)$ for $L \in \cl{F}$.
\end{definition}
This is the definition given in \cite[Section 2.7]{Mol88}. There is an intimate relation between Molino transverse equivalence of foliations, and Morita equivalence of their holonomy groupoids.
\begin{lemma}[{\cite[Corollary 3.17]{Miy23}}]
  \label{lem:6}
  If two regular foliations have Hausdorff holonomy groupoids, they are Molino transverse equivalent if and only if their holonomy groupoids are Morita equivalent.
\end{lemma}
The holonomy groupoid of Riemannian foliation is always Hausdorff (\cite[Example 5.8 (7)]{MM03}), so the following corollary is a direct consequence of Corollary \ref{cor:2} and Lemma \ref{lem:6}.
\begin{corollary}
Two Riemannian foliations have diffeomorphic leaf spaces if and only if they are Molino transverse equivalent. 
\end{corollary}

\begin{remark}
  \label{rem:4}
  Garmendia and Zambon \cite{GZ19} also propose a transverse equivalence for singular foliations, which they call \emph{Hausdorff Morita equivalence}. If two regular foliations (more generally, projective singular foliations) have Hausdorff holonomy groupoids, then they are Hausdorff Morita equivalent if and only if their holonomy groupoids are Morita equivalent (\cite[Proposition 3.39]{GZ19}).
\end{remark}

\printbibliography

\end{document}